\documentclass[a4paper,10pt,reqno]{amsart}
\usepackage{amsmath}
\usepackage{amsthm}

\usepackage[english]{babel} 
\usepackage[utf8]{inputenc}

\usepackage{parskip} 
\usepackage{cancel}

\usepackage[margin=1in]{geometry}

\usepackage{amssymb}

\usepackage{color}

\usepackage{url}

\newtheorem{thm}{Theorem}[section] 
\newtheorem{cor}[thm]{Corollary}
\newtheorem{lem}[thm]{Lemma}
\newtheorem{pro}[thm]{Proposition}
\newtheorem{den}[thm]{Definition}
\newtheorem{oss}[thm]{Remark}
\numberwithin{equation}{section}

\newcommand{\LL}{\mathrm{L}}

\begin{document}
	
\title[Smoothing effects for the filtration equation with different powers]{Smoothing effects for the filtration equation \\ with different powers}

\author {Alin Razvan Fotache, Matteo Muratori}

\address{Alin Razvan Fotache: Dipartimento di Matematica, Politecnico di Milano, Piazza Leonardo da Vinci 32, 20133 Milano, Italy}

\email {alinrazvan.fotache@polimi.it}

\address{Matteo Muratori: Dipartimento di Matematica ``F. Casorati'', Universit\`a degli Studi di Pavia, Via A. Ferrata 5, 27100 Pavia, Italy}

\email {matteo.muratori@unipv.it}

\begin{abstract}
We study the nonlinear diffusion equation $ u_t=\Delta\phi(u) $ on general Euclidean domains, with homogeneous Neumann boundary conditions. We assume that $ \phi^\prime(u) $ is bounded from below by $ |u|^{m_1-1} $ for small $ |u| $ and by $ |u|^{m_2-1} $ for large $|u|$, the two exponents $ m_1,m_2 $ being possibly different and larger than one. The equality case corresponds to the well-known porous medium equation. We establish sharp short- and long-time $ \LL^{q_0} $-$ \LL^\infty $ smoothing estimates: similar issues have widely been investigated in the literature in the last few years, but the Neumann problem with different powers had not been addressed yet. This work extends some previous results in many directions.
\end{abstract}


\keywords{Filtration equation; porous medium equation; Neumann problem; smoothing effects; asymptotic behaviour; Sobolev inequalities; Gagliardo-Nirenberg inequalities; Poincar\'e inequality; mean value; Moser iteration.}

\maketitle	
	
\section{Introduction}\label{sect: intro}
	
The present paper is devoted to the study of \emph{smoothing} and \emph{asymptotic} properties for solutions of the following \emph{filtration equation} with homogeneous \emph{Neumann} boundary conditions:	
\begin{equation}\label{pb}
	\begin{cases}
		u_t=\Delta\phi(u) & \textrm{in } \Omega\times \mathbb{R}^+ \, , \\
		\frac{\partial \phi(u)}{\partial n}=0 & \textrm{on } \partial\Omega\times \mathbb{R}^+ \, , \\
		u(0)=u_0  & \textrm{in } \Omega \, ,
	\end{cases}
\end{equation}
where $ \phi : \mathbb{R} \mapsto \mathbb{R} $ is a continuous and increasing function vanishing at zero, $ \Omega $ is a general domain of $ \mathbb{R}^N $ (not necessarily bounded or regular) and $ u_0 $ is an initial datum having suitable integrability properties that we shall specify below. Keeping in mind the widely studied case $ \phi(u)=|u|^{m-1}u $ (let $ m>1 $), we can also refer to \eqref{pb} as \emph{generalized porous medium equation}, in agreement with \cite{Vbook}. In fact we shall assume throughout, with the exception of Section \ref{sect: well}, that $\phi$ is $ C^1(\mathbb{R})$ and satisfies the following hypotheses: 
\begin{gather}  \label{cond-phi-1}
\phi(0) = 0 \, , \\ \label{cond-phi-2}
c_1 \, |u|^{m_1-1}  \le \phi^\prime(u) \quad \forall u: \, |u| \in [0,1] \, , \\  \label{cond-phi-3}
c_2 \, |u|^{m_2-1}  \le \phi^\prime(u) \quad \forall u: \, |u|>1  \, , 
\end{gather}
for some exponents $m_1, m_2>1$ and positive constants $ c_1,c_2 $. In other words, if we think of \eqref{cond-phi-2}--\eqref{cond-phi-3} as equalities, we are allowing \eqref{pb} to be like a porous medium equation with exponent $m_1$ where the solution is small and like a porous medium equation with another exponent $m_2$ where the solution is large. We shall see that $ m_2 $ is associated with short-time behaviour, whereas $m_1$ is associated with long-time asymptotics. It turns out that, to our purposes, the only requirements that count are bounds from \emph{below} on $ \phi^\prime $ like \eqref{cond-phi-2}--\eqref{cond-phi-3}, so that actually $ \phi $ may significantly deviate from powers. 

Recently, as concerns the straight porous-medium nonlinearity $ \phi(u)=|u|^{m-1}u $, in \cite[Theorem 3.2]{GM13} it has been proved that, if $ \Omega $ is bounded and regular and the spatial dimension $N$ is greater than or equal to $3$, the $ \LL^{q_0} $-$ \LL^{\infty} $ smoothing effect 
\begin{equation}\label{stimaPrecisaMezziPorosi}
\left\| u(t) \right\|_{\infty} \leq K \left(t^{-\frac{N}{2q_0+N(m-1)}} \left\| u_0 \right\|_{q_0}^{\frac{2q_0}{2q_0+N(m-1)}} + \left\| u_0 \right\|_{q_0} \right) \quad \forall t>0 
\end{equation}
holds for all $ q_0 \in [1,\infty) $ and a suitable $K>0$. As for long-time asymptotics, such estimate can be improved depending on whether $ \overline{u}_0=0 $ or $ \overline{u}_0 \neq 0 $, where $ \overline{u}_0 $ is the mean value of the initial datum. In the case $ \overline{u}_0=0 $, it is shown in \cite[Theorem 4.1]{GM13} that
\begin{equation}\label{stimaPrecisaMezziPorosi-1}
\left\| u(t) \right\|_{\infty} \leq K_1 \, t^{-\frac{N}{2q_0+N(m-1)}} \left( K_2 \, t + \left\| u_0 \right\|_{q_0}^{1-m} \right)^{-\frac{2q_0}{(m-1)[2q_0+N(m-1)]}} \quad \forall t>0 
\end{equation} 
holds for $ q_0 \in [1,\infty) $ and suitable $ K_1,K_2>0 $, whereas in the case $ \overline{u}_0 \neq 0 $ \cite[Theorem 4.3]{GM13} establishes that
\begin{equation}\label{stimaPrecisaMezziPorosi-2}
\left\| u(t) - \overline{u}_0 \right\|_{\infty} \leq G \, e^{-\frac{m\left| \overline{u}_0 \right|^{m-1}}{C_P^2}\,t} \quad \forall t \ge 1 \, ,
\end{equation}
where $ G $ is a suitable positive constant and $ C_P>0 $  is the \emph{best constant} in the \emph{Poincar\'e} inequality
\begin{equation}\label{dis-poin}
\left\| f-\overline{f} \right\|_2 \le C_P \left\| \nabla f \right\|_2 \quad \forall f \in H^1(\Omega) \, , \quad \overline{f}:=\frac{\int_\Omega f(x) \, dx}{|\Omega|} \, .
\end{equation}  
The above estimates were obtained by only exploiting the standard \emph{Sobolev} inequality
\begin{equation}\label{Sob}
\left\| f \right\|_{2^\star} \le C_\star \left( \left\| \nabla f \right\|_{2} + \left\| f \right\|_2 \right) \quad \forall f \in H^1(\Omega) \, , \quad 2^\star := \frac{2N}{N-2} 
\end{equation}  
and \eqref{dis-poin}, the latter in order to get \eqref{stimaPrecisaMezziPorosi-1} and \eqref{stimaPrecisaMezziPorosi-2}. Indeed such results could be extended to the case of \emph{weighted} porous medium equations (or to rougher domains), subject to the validity of the analogues of \eqref{Sob} and \eqref{dis-poin} in the corresponding framework (see \cite[Section 5]{GM13}). Afterwards, it was proved in \cite{GM14} that \eqref{stimaPrecisaMezziPorosi}--\eqref{stimaPrecisaMezziPorosi-2} are still true in low dimensions, up to taking advantage of \emph{Gagliardo-Nirenberg} inequalities. We stress that both in \cite{GM13} and \cite{GM14} the hypothesis $ |\Omega|<\infty $ was essential. 

Here we shall assume that $ \Omega $ is any domain of $ \mathbb{R}^N $ that supports the following \emph{Gagliardo-Nirenberg-Sobolev} inequalities:
\begin{equation}\label{NGN}
\left\| f \right\|_r \le C_S \left( \left\| \nabla{f} \right\|_2 + \left\| f \right\|_2 \right)^{\vartheta(s,r,N)} \left\| f \right\|_s^{1-\vartheta(s,r,N)} \quad \forall f \in H^1(\Omega) \cap \LL^s(\Omega) 
\end{equation}
with
\begin{equation}\label{NGN-parameters-theta}
\vartheta(s,r,N) := \frac{2N\,(r-s)}{r\,[2N-s(N-2)]} \, ,
\end{equation}
where if $ N=1 $ or $ N=2 $ we suppose that $ r,s $ can vary subject to
\begin{equation}\label{NGN-parameters-r-s-1}
 0 < s < r < \infty \, ,
\end{equation} 
whereas in the case $ N\ge 3 $ we suppose they can vary subject to
\begin{equation}\label{NGN-parameters-r-s-2}
0 < s < r \le 2^\star \quad \textrm{or} \quad 2^\star \le r < s < \infty \, .
\end{equation}
The positive constant $C_S$ is required to be bounded independently of $ r,s $ as long as the latter range in compact subsets of $ (0,\infty) $. By means of Young's inequality, it is straightforward to deduce from \eqref{NGN} the validity of
\begin{equation}\label{NGN-bis}
\left\| f \right\|_r \le C_S \left( \left\| \nabla{f} \right\|_2 + \left\| f \right\|_s \right)^{\vartheta(s,r,N)} \left\| f \right\|_s^{1-\vartheta(s,r,N)} \quad \forall f \in H^1(\Omega) \cap \LL^s(\Omega) 
\end{equation}
under the additional constraint $ s \le 2 $ (for another constant $ C_S $ as above that we do not relabel), which will turn out to be useful in the sequel.  

Inequalities \eqref{NGN} are not chosen by chance. In the seminal paper \cite{BCLS} it was established (in more abstract contexts actually) that the validity of \eqref{NGN} for a \emph{single} pair $ (r,s) $ is equivalent to the validity of the whole family, i.e.~of \eqref{NGN} itself for \emph{all} $r,s$ complying with \eqref{NGN-parameters-r-s-1} or \eqref{NGN-parameters-r-s-2}, depending on the spatial dimension. Note that such result readily follows by interpolation in the special case where one picks the Sobolev inequality \eqref{Sob} as a representative of the family (provided $ N \ge 3 $). Furthermore, it is well known that \eqref{NGN} holds for regular, compactly supported functions on $ \mathbb{R}^N $ (namely in $ \mathcal{D}(\mathbb{R}^N) $) with no additional $ \LL^2 $ norm in the right-hand side, and the latter are \emph{equivalent} to a precise power-rate time decay for the associated \emph{heat kernel} (see e.g.~\cite[Chapter 2]{Davies} and Remark \ref{rem: dirichlet} below). As a consequence, they hold in the form \eqref{NGN} on all Euclidean domains having the \emph{extension} property; more in general, at least in dimension $ N \ge 3 $, they hold on Euclidean domains complying with the \emph{cone} condition. For such results, we refer the reader e.g.~to the monograph \cite[Chapters 4 and 5]{AF}. 
See also the classical, celebrated papers \cite{Gag1,Nir} for a thorough analysis on the validity of this kind of inequalities in the Euclidean setting.  

On the other hand, the Poincar\'e inequality \eqref{dis-poin} only makes sense on finite-measure domains and can actually be shown to hold on any such domain supporting \eqref{NGN}, as we shall prove in Proposition \ref{pro-imp} below. This is strictly related to the \emph{local} compactness of the embedding of $ H^1(\Omega) $ into $ \LL^2(\Omega) $ due to Rellich's Theorem.

%

\medskip 

\textbf{Previous results.} There is a huge literature concerned with the topics addressed here. Accordingly, with no claim at all for completeness, below we quote some of the most relevant papers.

The Neumann problem for the porous medium equation, before \cite{GM13,GM14}, had not widely been investigated. In \cite{A1} some $ \LL^\infty $ estimates for a similar problem with a reaction term were proved (though without establishing $ \LL^{q_0} $-$ \LL^\infty $ regularizing effects). A remarkable paper was then \cite{AR}, where the authors obtained (almost) sharp asymptotic results, as $ t \to \infty $, on bounded regular domains and for initial data in $ \LL^\infty(\Omega) $. In particular, they showed how to handle separately the zero-mean and nonzero-mean cases. In \cite{BG05} smoothing effects first appeared for the Neumann problem (by means of pure functional inequalities), but they turned out not to be fully sharp as pointed out in \cite{GM13}. 

As for long-time asymptotics, stabilization towards the mean value, in agreement with \eqref{stimaPrecisaMezziPorosi-1}--\eqref{stimaPrecisaMezziPorosi-2}, is not a new phenomenon: see e.g.~\cite{EK05} for heat-type equations with density vanishing at infinity in one dimension, and \cite{EKT} for similar results (in higher dimensions) where the degeneracy lies in the diffusion coefficients. In \cite{DGGW} and \cite{GMP13} convergence to the mean value was studied for weighted porous medium equations, by means of Poincar\'e-type inequalities. 

The literature related to smoothing effects in the case of \emph{Dirichlet}-type problems (or problems on the whole space) is more extensive: see \cite{Vsmooth} as a comprehensive reference. We refer to \cite{BGV} and \cite{GM16} for $ \LL^{q_0} $-$ \LL^\infty $ smoothing effects on \emph{Cartan-Hadamard} manifolds, in the fast-diffusion ($ m<1 $) and porous-medium case, respectively. As regards weighted porous medium equations, in \cite{GMP13} $ \LL^{q_0} $-$ \LL^p $ smoothing effects (with $ p \in (q_0,\infty) $) were established by only assuming a (spectral-gap) Poincar\'e inequality, which in general prevents $ \LL^\infty $ regularization. As for the \emph{fractional} porous medium equation on Euclidean space, we quote \cite{DQRV} and \cite{GMPu}, where fractional Gagliardo-Nirenberg-type (or Nash-type) inequalities were used. In \cite{BV15}, the same equation was considered on domains with homogeneous Dirichlet boundary conditions, and smoothing effects were proved by means of smart Green-function techniques. The $p$-Laplacian equation was then addressed in \cite{G}, through functional-analytic arguments involving \emph{logarithmic} Sobolev inequalities, and in \cite{BCG}, showing optimal convergence to the mean value (on compact manifolds without boundary). In \cite{BG06} the authors analysed doubly nonlinear equations, obtaining sharp smoothing effects still by means of a differential method that exploits logarithmic Sobolev inequalities. 

Actually, smoothing estimates for Neumann problems (and general equations of $p$-Laplacian type) are also considered in \cite{ACLT}, but on domains for which there hold functional inequalities which make the solution behave in a similar way to the Dirichlet case. Doubly nonlinear equations on domains narrowing at infinity are the main subject of \cite{AT00}, even though the geometry of the domains at hand makes again the functional setting closer to a Dirichlet-type one. For similar results on noncompact manifolds (by means of \emph{Faber-Krahn} inequalities), see e.g.~\cite{AT15}. Dirichlet problems on unbounded domains, for the porous medium equation, are then analysed in \cite{AT14} through \emph{harmonic} functions that play a role in weighted Gagliardo-Nirenberg inequalities (with no additional $ \LL^2 $ term), so as to obtain smoothing effects with respect to the corresponding weighted norms. 

An interesting alternative approach, which consists in obtaining preliminary estimates on ``truncated'' solutions and then pass to the limit, was exploited in \cite{Po09} to prove smoothing and decay estimates for $p$-Laplacian-type equations and Dirichlet-type problems; further developments of such an approach were then carried out in \cite{Po11} under milder conditions on $p$ and in \cite{Po15} to deal with more general equations.

We finally quote \cite{ST}, where smoothing effects are obtained for \emph{systems} of porous-medium-type equations, then generalized to the doubly nonlinear case in \cite{TV}.

In the recent paper \cite{CH}, a global theory of smoothing effects for nonlinear semigroups has been set up, which encompasses many of the equations discussed above. The authors proceed by means of time discretization and exploit suitable Gagliardo-Nirenberg inequalities. Their results hold in very abstract frameworks. Nevertheless, we point out that the problems studied here are not included in such theory, both at the level of functional inequalities (due to the additional $\LL^2$ norm in the r.h.s.~of \eqref{NGN}) and at the level of the nonlinearity we consider, which is not necessarily a single power. Moreover, their estimates are mostly significant for short times, namely they do not investigate the validity of bounds of the type of \eqref{stimaPrecisaMezziPorosi-1} or \eqref{stimaPrecisaMezziPorosi-2}.

Let us now turn to the filtration equation. The latter was addressed by several authors: first of all we quote two papers that have been seminal with respect to many aspects, namely \cite{DhK} and \cite{AD}. The weighted case (i.e.~with a density) on Euclidean space was studied in \cite{Ei90} (existence, uniqueness and basic estimates), for a rather general $ \phi $; similar issues were discussed in \cite{EK94} on exterior domains. As one of the first papers concerned with the filtration equation we quote \cite{K76}, which deals with the one-dimensional filtration equation with respect to asymptotics via self-similar solutions, under particular conditions on $ \phi $. Then, in \cite{KR82}, the asymptotics for the same equation with finite-mass densities was investigated (proving convergence to the mean value), while in \cite{GHP} (dimension one and two) the authors analysed support and blow-up properties. In general, when the density decays sufficiently fast at infinity, nontrivial well-posedness issues arise, which were studied in \cite{GMPu14}. 

Even \emph{nonlocal} versions of the filtration equation have recently been investigated. In \cite{BV16} the authors address the Dirichlet problem for a very general equation which covers both the local and the nonlocal case, under suitable assumptions on the \emph{Green} function associated with the operator considered. In particular, up to slightly stronger requirements on $ \phi $ (see Section 2 there), they obtain smoothing effects (see Corollary 6.3 there) analogous to those discussed in Remark \ref{rem: dirichlet}: to the best of our knowledge, this is the first paper dealing with a function $\phi$ that is allowed to have two different power-type behaviours at zero and at infinity (though restricted to Dirichlet problems). As for the problem on Euclidean space, in \cite{VDQR} fine regularity results have been shown in the case where the nonlocal operator is the standard fractional Laplacian, whereas in \cite{DQR} similar properties have been studied for operators with rougher kernels.

For a wide dissertation on filtration equations, we also refer the reader to \cite{DK}, even if the analysis there is mostly concerned with regularity properties and estimates for nonnegative local solutions, solutions on the whole Euclidean space or solutions of the Dirichlet problem on regular domains, especially when $ \phi $ is trapped between two powers at infinity. 

\medskip

\textbf{Organization of the paper.} In Section \ref{sect: res} we present our main results. Theorem \ref{thm01} establishes smoothing effects for general $ \LL^{q_0} $ data, providing an estimate that is the analogue of \eqref{stimaPrecisaMezziPorosi} with $ m=m_2 $ and $ m=m_1 $ for short and long times, respectively. The corresponding proof is given in Section \ref{sect: short}, and proceeds by means of nontrivial modifications of well-established Moser iterations (which in fact go back to \cite{Mo64,Mo71}). Theorem \ref{smoothing-asym-zero-1} then yields a better estimate (the analogue of \eqref{stimaPrecisaMezziPorosi-1}) under the additional assumption that the initial datum, and therefore the solution, has zero mean. The corresponding proof is given in Section \ref{sect: long-zero}. Theorem \ref{thm-asym-nonzero} deals with the case of data, and solutions, having nonzero mean. In view of the $ \LL^\infty $ smoothing effect, we can get the analogue of \eqref{stimaPrecisaMezziPorosi-2}, in the sense that the exponential decay is the one predicted by linearization about the mean value: we prove it in Section \ref{sect: long-non}, and the argument requires a little more regularity on $ \phi^\prime $ (see also Remark \ref{mod-cont}). Sharpness of our estimates, mainly as regards short times, is extensively discussed in Section \ref{sec:sharp}. Finally, Section \ref{sect: well} is devoted to providing basic well-posedness results for problem \eqref{pb}, which however need to be treated cautiously due to the generality of our assumptions on $ \Omega $ and $\phi$. 

\section{Statements of the main results}\label{sect: res} 

We describe here our results concerning smoothing and asymptotic estimates for solutions of \eqref{pb}, under suitable hypotheses on $ \Omega $ that only involve the validity of functional inequalities, as discussed in the Introduction. A precise meaning to the concept of ``solution'' will be given in Section \ref{sect: well}, see in particular Remark \ref{rem: uniq} there.

\begin{thm}[Smoothing]\label{thm01}
Let $ \Omega \subset \mathbb{R}^N $ be a domain that supports the Gagliardo-Nirenberg-Sobolev inequalities \eqref{NGN}. Let $u$ be the solution of \eqref{pb} corresponding to an initial datum $ u_0 \in \LL^1(\Omega) \cap \LL^{q_0}(\Omega) $ with $q_0\in [1,\infty)$, where $\phi \in C^1(\mathbb{R}) $ is any nonlinearity complying with \eqref{cond-phi-1}--\eqref{cond-phi-3}. Then the following smoothing estimate holds:
\begin{equation}\label{thm01-smooth-est}
	\Vert u(t)\Vert_{\infty}\leq 
	\begin{cases}
	K \left( t^{-\frac{N}{2q_0+N(m_2-1)}} \, \Vert u_0\Vert_{q_0}^{\frac{2q_0}{2q_0+N(m_2-1)}} + \Vert u_0\Vert_{q_0} \right) & \forall t \in\left( 0 , \left\| u_0 \right\|_{q_0}^{\frac{2q_0}{N}} \right) , \\
	K \left( t^{-\frac{N}{2q_0+N(m_1-1)}} \, \Vert u_0\Vert_{q_0}^{\frac{2q_0}{2q_0+N(m_1-1)}} + \Vert u_0\Vert_{q_0} \right) & \forall t\geq \Vert u_0\Vert_{q_0}^{\frac{2q_0}{N}}  ,
\end{cases}
\end{equation} 
where $K$ is a positive constant depending only on the spatial dimension $ N $, the constants $ m_1 , m_2 , c_1 , c_2 $ in the lower bounds \eqref{cond-phi-2}--\eqref{cond-phi-3} and the constant $ C_S $ in the Gagliardo-Nirenberg-Sobolev inequalities \eqref{NGN}.
\end{thm}

In the case of domains with finite measure 
the above result can be improved, especially as concerns long-time asymptotics. In this regard, it is crucial to treat separately data (and therefore solutions) with zero and nonzero mean.
\begin{thm}[Smoothing and asymptotics, $ \overline{u}_0=0 $]\label{thm-asym-zero}
Let the hypotheses of Theorem \ref{thm01} be fulfilled. Suppose moreover that $ \Omega $ is of finite measure
and that $ \overline{u}_0 = 0 $. Then the following estimates hold:
\begin{equation}\label{smoothing-asym-zero-1}
\Vert u(t)\Vert_{\infty} \leq 
\begin{cases}
K \, t^{-\frac{N}{2q_0+N(m_2-1)}} \left\| u_0 \right\|_{q_0}^{\frac{2q_0}{2q_0+N(m_2-1)}} & \forall t \in \left( 0, \| u_0 \|^{\frac{2q_0}{N}}_{q_0} \right) , \\ 
K \, t^{-\frac{N}{2q_0+N(m_1-1)}} \left\| u_0 \right\|_{q_0}^{\frac{2q_0}{2q_0+N(m_1-1)}} & \forall t \in \left[ \| u_0 \|^{\frac{2q_0}{N}}_{q_0} , 1  \right] , \\
K \, t^{-\frac{N}{2q_0+N(m_1-1)}} \left( t + \| u_0 \|_{q_0}^{1-m_1} \right)^{-\frac{2q_0}{(m_1-1)[2q_0 + N(m_1-1)]}} & \forall t > 1 \, ,
\end{cases}
\end{equation}  
$$ \textrm{for all } u_0 : \, \| u_0 \|_{q_0} \le 1 \, , $$ 
and
\begin{equation}\label{smoothing-asym-zero-2}
	\Vert u(t)\Vert_{\infty}\leq 
	\begin{cases}
K \, t^{-\frac{N}{2q_0+N(m_2-1)}} \left( t + \| u_0 \|_{q_0}^{1-m_2} \right)^{-\frac{2q_0}{(m_2-1)[2q_0 + N(m_2-1)]}} &  \forall t \in \left( 0, 1 \right) , \\
	K \, t^{-\frac{1}{m_1-1}} & \forall t >1 \,  ,
\end{cases}
\end{equation}
$$ \textrm{for all } u_0 : \, \| u_0 \|_{q_0} > 1 \, , $$
where $ K $ is a positive constant depending only on the domain $ \Omega $ and the constants $m_1 , m_2 , c_1 , c_2 $ in the lower bounds \eqref{cond-phi-2}--\eqref{cond-phi-3}. 
\end{thm} 

\begin{thm}[Asymptotics, $ \overline{u}_0 \neq 0 $]\label{thm-asym-nonzero}
Let the hypotheses of Theorem \ref{thm01} be fulfilled, with the additional assumption $ \phi \in C^2(\mathbb{R} \setminus \{0\}) $. Suppose moreover that $ \Omega $ is of finite measure
and that $ \overline{u}_0 \neq 0 $. Then the following estimate holds: 
\begin{equation}\label{est-smooth}
\left\| u(t)-\overline{u}_0 \right\|_\infty \le G \, e^{-\frac{\phi^\prime\!(\overline{u}_0)}{C_P^2} \, t} \left\| u_0 - \overline{u}_0 \right\|_1 \quad \forall t \ge 1 \, , 
\end{equation}
where $C_P$ is the best constant in \eqref{dis-poin} and $G$ is a positive constant depending only on $ \| u_0 \|_1 , |\overline{u}_0| , \phi , \Omega $, which is increasing w.r.t.~$ \| u_0 \|_1 $ and locally bounded w.r.t.~$ |\overline{u}_0|>0 $.
\end{thm}

Theorems \ref{thm01}, \ref{thm-asym-zero} and \ref{thm-asym-nonzero} will be proved in Sections \ref{sect: short}, \ref{sect: long-zero} and \ref{sect: long-non}, respectively.

\begin{oss}[Possible generalizations]\label{rem: gen}\rm
The strategies of proof we employ, based on Moser iterations and semigroup arguments, are quite general, and in particular can be straightforwardly adapted to deal with analogous problems on Riemannian manifolds, in the spirit e.g.~of \cite{BG05,BGV,GM16}, or problems with weights, in the spirit e.g.~of \cite{GM13,GMP13,GM14}. The only essential hypothesis we need is the validity of functional inequalities of the type of \eqref{NGN} or \eqref{dis-poin} (in the case of Riemannian manifolds, see the classical reference \cite{Heb}). Note that in more abstract settings where Rellich's Theorem is \emph{a priori} not guaranteed, the Poincar\'e inequality should be required explicitly (Proposition \ref{pro-imp} below may fail). However, in order not to divert the discussion from the main topics, we preferred to work in the standard framework of Euclidean domains. 
\end{oss}

\begin{oss}[The porous medium equation]\rm
It is worth mentioning that Theorem \ref{thm01} holds in the straight porous-medium case as well (i.e.~when $ m_1=m_2=m>1 $) and therefore allows us to improve on the short-time results of \cite{GM13}, in the sense that we succeed in removing the hypothesis of finiteness of $ |\Omega| $ (through more direct techniques actually, see Section \ref{sect: short}).
\end{oss}

\subsection{Sharpness of the estimates}\label{sec:sharp}
In the following we address the question of optimality of our smoothing estimates. We restrict the analysis to short times (that is to \eqref{thm01-smooth-est} as $ t \to 0 $), because for long times, due to $ \LL^\infty $ regularization, the evolution falls within the framework of single-power nonlinearities (i.e.~of the porous-medium type), where optimality was thoroughly discussed in \cite{GM13} (see Remark 4.2 and Proposition 4.5 there). Furthermore, in order not to weigh the discussion down with too many technicalities, we shall deal with the case $ q_0=1 $ only, which is by the way the most significant one in the literature.

The basic idea lying behind optimality is quite simple and is borrowed from \cite[Remark 3.4]{GM13}, \cite[Section 2.2]{Vsmooth}: as we are dealing with a diffusion equation of heat-type structure, it seems reasonable to conjecture that the worst possible initial datum is a \emph{Dirac delta}, which by its nature involves the behaviour of $ \phi(u) $ at infinity. Since the latter at infinity is approximately $ u^{m_2} $ (at worst), we expect that the best smoothing estimate one can get is the one associated with such power, namely \eqref{thm01-smooth-est} for small $t$. The rigorous justification of such an argument is however nontrivial, and this is the purpose of the sequel of this section.

Given any domain $ \Omega \subset \mathbb{R}^N $, $ m_1,m_2 > 1 $ and $ c_1,c_2 > 0 $, we say that an estimate of the type of $ \| u(t) \|_{\infty} \le \mathcal{S}(t,\| u_0 \|_1) $ is \emph{sharp} for short times if there exists $ \phi \in C^1(\mathbb{R}) $ fulfilling \eqref{cond-phi-1}--\eqref{cond-phi-3} such that the corresponding solutions to \eqref{pb} comply with the following lower bound:
\begin{equation}\label{eq:def-sharpness}
 \inf_{M>0} \, \liminf_{t \to 0} \, \sup_{u_0 \in \LL^1(\Omega): \, \| u_0 \|_1 = M } \, \frac{\| u(t) \|_{\infty}}{\mathcal{S}(t,\| u_0 \|_1)} > 0 \, .
\end{equation} 
In other words we require that, up to multiplicative constants, the estimate captures the precise behaviour as $ t\to 0 $ of the worst possible solution having a given mass (think e.g.~of positive solutions). Then the ratio between the behaviour of such solution and the estimate must be independent of the mass. 

We shall prove that \eqref{eq:def-sharpness} does hold when $ \mathcal{S}(t,\| u_0 \|_1) $ is the right-hand side of \eqref{thm01-smooth-est} at $ q_0=1 $. First of all, let us fix any $ \phi \in C^1(\mathbb{R}) $ with $ \phi^\prime>0 $ satisfying 
\begin{equation}\label{eq:choice-phi} 
\phi(u)=
\begin{cases}
u^{m_1} & \forall u: \, |u| \in \left[0 , \frac 1 2 \right] , \\
u^{m_2} & \forall u: \, |u| \ge 2 \, ,
\end{cases}
\end{equation}
where $ u^m:=|u|^{m-1}u $. It is plain that such a function fulfils \eqref{cond-phi-1}--\eqref{cond-phi-3} for suitable $ c_1,c_2>0 $. Actually $ c_1$ and $c_2 $ are supposed to be given data: in any case, it is enough to multiply the above $ \phi $ by a large enough constant so that \eqref{cond-phi-1}--\eqref{cond-phi-3} are satisfied with the given values of $ c_1 $ and $c_2$. For simplicity, we shall keep $ \phi $ as in \eqref{eq:choice-phi}, and we leave it to the reader to check that a multiplication by a constant adds no significant difficulty in the discussion below. Now we pick $ {\phi_\ast} \in C^1(\mathbb{R}) $, with $ \phi_\ast^\prime>0 $, such that
\begin{equation}\label{eq:ast-phi}
\phi_\ast(u)=
\begin{cases}
u^{m_2} & \quad \forall u: \, |u| \in \left[ 0, \frac{1}{16} \right] , \\
\phi(u) & \quad \forall u: \, |u| \ge \frac18 \, . \\
\end{cases} 
\end{equation}
Let $ \mathcal{U}_\ast $ be the ``Barenblatt'' solution to the Cauchy problem
\begin{equation}\label{pb-modified}
	\begin{cases}
		u_t=\Delta\phi_\ast(u) & \textrm{in } \mathbb{R}^N \times \mathbb{R}^+ \, , \\
		u(0)=\delta_{x_0} & \textrm{in } \mathbb{R}^N \, ,
	\end{cases}
\end{equation}
where $ \delta_{x_0} $ is the Dirac delta centred at some $ x_0 \in \Omega $. Because $ \phi_\ast(u) $ behaves like $ u^{m_2} $ both at zero and at infinity, one can show (e.g.~by standard barrier arguments) that $ \mathcal{U}_\ast $ enjoys approximately the same scaling properties as the very Barenblatt solution associated with $ \phi(u) \equiv u^{m_2} $. In particular, 
\begin{equation}\label{eq:scal-approx}
\| \mathcal{U}_\ast (t) \|_1 = 1 \, , \quad \infty > \| \mathcal{U}_\ast(t) \|_\infty \ge K_0 \, t^{-\frac{N}{2+N(m_2-1)}} \quad \textrm{and} \quad  \operatorname{supp} \, \mathcal{U}_\ast(\cdot,t) \subset B_{R_0 \, t^{\frac{1}{2+N(m_2-1)}}}^{(x_0)} \quad \forall t > 0
\end{equation} 
for suitable positive constants $ K_0 $ and $R_0$, where $ B_r^{(x_0)} $ is the ball of radius $ r>0 $ centred at $ x_0 $. Given $ \ell>0 $, let us then consider the Barenblatt solution $ \mathcal{U}_{\ell} $ to
\begin{equation}\label{pb-barenblatt-1}
	\begin{cases}
		u_t=\Delta u^{m_1} & \textrm{in } \mathbb{R}^N \times \mathbb{R}^+ \, , \\
		u(0)= \ell \, \delta_{x_0} & \textrm{in } \mathbb{R}^N \, ,
	\end{cases}
\end{equation}
whose profile is self-similar and radially decreasing w.r.t.~$ |x-x_0| $ (see e.g.~\cite[Remark 3.4]{GM13}). We can therefore select $ t_0>0 $ and $ \ell>0 $ in such a way that 
\begin{equation}\label{pb-barenblatt-2}
\| \mathcal{U}_{\ell}(t_0) \|_\infty = \frac{1}{4} \quad \textrm{and} \quad \operatorname{supp} \,  \mathcal{U}_{\ell}(\cdot,t_0) \subset \Omega \, .
\end{equation}
We finally combine $ \mathcal{U}_\ast  $ and $ \mathcal{U}_\ell $ by setting
\begin{equation}\label{eq-uhat}
\hat{u}(x,t) := \max\left\{ \mathcal{U}_\ast(x,t) \, , \, \mathcal{U}_{\ell}(x,t+t_0) \right\} \quad \forall (x,t) \in \mathbb{R}^N \times \mathbb{R}^+ \, .
\end{equation}
Thanks to \eqref{pb-barenblatt-1}--\eqref{pb-barenblatt-2}, the continuity of $ \mathcal{U}_{\ell} $ for positive times, \eqref{eq:ast-phi}--\eqref{eq:scal-approx} and \eqref{eq:choice-phi}, it is apparent that there exists $ t_s > 0 $ such that, for any given $ \tau \in (0,t_s) $, the function $ (x,t) \mapsto \hat{u}(x,t+\tau) $ is a compactly supported \emph{subsolution} to problem \eqref{pb} with initial datum $ u_0(x)=\hat{u}(x,\tau) $ for all $ t \in (0,t_s-\tau) $. Moreover, since the total mass of $ \mathcal{U}_{\ell}(\cdot,t) $ is also preserved in time and $ \mathcal{U}_{\ast}(\cdot,\tau) $ tends to a Dirac delta as $ \tau \to 0 $,   
\begin{equation}\label{pb-barenblatt-3}
\lim_{\tau \to 0} \| \hat{u}(\tau) \|_1 = 1+\ell  \quad \textrm{and} \quad \| \hat{u}(\tau) \|_1 \le 1+\ell \quad \forall \tau>0 \, .
\end{equation}
As $ \phi $ is independent of $ x $ and $ \hat{u}(\cdot,t) $ is compactly supported in $ \Omega $ for all $ t \in (0,t_s) $, by parabolic scaling we know there exists $ \varepsilon>0 $ such that $ \hat{u}_{\lambda,\tau}(x,t):=\hat{u}(x_0+\lambda(x-x_0) ,\lambda^2 (t + \tau)) $ is still a compactly supported subsolution to problem \eqref{pb} with initial datum $ u_0(x)=\hat{u}(x_0+\lambda(x-x_0),\lambda^2\tau) $ for all $ t \in (0,t_s/\lambda^2-\tau ) $ and positive $ \lambda > 1-\varepsilon $ subject to $ \lambda^2 < t_s/\tau $. Hence, in view of \eqref{pb-barenblatt-3} and the continuity of both $ \mathcal{U}_{\ast} $ and $ \mathcal{U}_{\ell} $ for positive times, it is possible to choose two positive sequences $ \lambda_n \to 1 $ and $ \tau_n \to 0 $ so that 
\begin{equation}\label{pb-barenblatt-4}
\| \hat{u}_{\lambda_n,\tau_n}(0) \|_1 = 1+\ell  \, , \quad  \lambda_n^2 < \frac{t_s}{\tau_n}  \quad \forall n \in \mathbb{N} \, .
\end{equation} 
So, by virtue of the middle inequality in \eqref{eq:scal-approx}, \eqref{pb-barenblatt-4} and recalling estimate \eqref{thm01-smooth-est}, we have:  
\begin{equation}\label{pb-barenblatt-5}
\frac{\| \hat{u}_{\lambda_n,\tau_n}(t) \|_\infty}{\mathcal{S}(t,\| \hat{u}_{\lambda_n,\tau_n}(0) \|_1)} 
\ge \frac{K_0 \left[\lambda_n^2 \left( 1+\frac{\tau_n}{t} \right) \right]^{-\frac{N}{2+N(m_2-1)}} }{K \left[ (1+\ell)^{\frac{2}{2+N(m_2-1)}} + (1+\ell) \, t^{\frac{N}{2+N(m_2-1)} \,} \right]} \quad \forall t \in \left( 0 , \left( {t_s}/{\lambda_n^2} - \tau_n \right) \wedge (1+\ell)^{\frac 2 N}  \right) .
\end{equation} 
Upon letting first $ n \to \infty $ and then $ t \to 0 $ in \eqref{pb-barenblatt-5}, we can therefore infer the validity of (note that $ \hat{u}_{\lambda_n,\tau_n} $ is a subsolution) 
\begin{equation}\label{pb-barenblatt-6}
\liminf_{t \to 0} \, \sup_{u_0 \in \LL^1(\Omega): \, \| u_0 \|_1 = 1+\ell } \, \frac{\| u(t) \|_{\infty}}{\mathcal{S}(t,\| u_0 \|_1)} \ge \frac{K_0}{K \, (1+\ell)^{\frac{2}{2+N(m_2-1)}} } \,  .
\end{equation} 
In order to complete the argument and prove \eqref{eq:def-sharpness}, we need to show that \eqref{pb-barenblatt-6} is in fact independent of the particular value of the mass $ 1+\ell $. To this end, let $ \alpha \in (0,1) $ and consider the functions 
$$ 
\hat{u}_{\lambda_n,\tau_n,\alpha}(x,t) :=  \hat{u}_{\lambda_n,\tau_n}\!\left(x_0+\alpha^{-\frac{1}{N}}(x-x_0) , \alpha^{-\frac{2}{N}}t\right) \quad \forall (x,t) \in \mathbb{R}^N \times \mathbb{R}^+ \, ,
$$ 
which by parabolic scaling are subsolutions to \eqref{pb} for all $ t \in (0,\alpha^{2/N} (t_s/\lambda_n^2-\tau_n) ) $ and they satisfy $ \| \hat{u}_{\lambda_n,\tau_n,\alpha}(0) \|_1 = \alpha(1+\ell) $ (recall \eqref{pb-barenblatt-4}). As a consequence, thanks again to the middle inequality in \eqref{eq:scal-approx}, by reasoning as above we deduce that 
\begin{equation}\label{pb-barenblatt-7} 
\begin{gathered}
\sup_{u_0 \in \LL^1(\Omega): \, \| u_0 \|_1 = \alpha(1+\ell) } \, \frac{\| u(t) \|_{\infty}}{\mathcal{S}(t,\| u_0 \|_1)}
\ge \frac{K_0}{K \left[ (1+\ell)^{\frac{2}{2+N(m_2-1)}} + \alpha^{\frac{N(m_2-1)}{2+N(m_2-1)}} \, (1+\ell)\, t^{\frac{N}{2+N(m_2-1)}} \right]} \\
 \forall t \in \left( 0 , \alpha^{\frac 2 N} \! \left[ t_s \wedge (1+\ell)^{\frac 2 N} \right] \right) .
\end{gathered}
\end{equation} 
As a final step, we let $ \beta>1 $ and set
$$ \mathcal{U}_{\ast,\beta}(x,t) := \mathcal{U}_\ast\!\left(x_0+\beta^{-\frac{1}{N}} (x-x_0) , \beta^{-\frac{2}{N}} t\right) \quad \forall (x,t) \in \mathbb{R}^N \times \mathbb{R}^+ \, ; $$ 
such functions, still as a consequence of parabolic scaling, are nothing but the solutions to \eqref{pb-modified} with $ \delta_{x_0} $ replaced by  $ \beta \delta_{x_0} $. Hence, from \eqref{eq:scal-approx}, for all $ t>0 $ there follows $ \| \mathcal{U}_{\ast,\beta} (t) \|_1 = \beta $,
\begin{equation}\label{eq:scal-approx-1}
\infty > \| \mathcal{U}_{\ast,\beta}(t) \|_\infty \ge K_0 \, \beta^{\frac{2}{2+N(m_2-1)}} \, t^{-\frac{N}{2+N(m_2-1)}} \quad \textrm{and} \quad  \operatorname{supp} \, \mathcal{U}_{\ast,\beta}(\cdot,t) \subset B_{R_0 \, \beta^{\frac {m_2-1}{2+N(m_2-1)} } t^{\frac{1}{2+N(m_2-1)}}}^{(x_0)} \, .
\end{equation}
If we repeat the arguments that led to \eqref{pb-barenblatt-5} and \eqref{pb-barenblatt-6}, up to the replacement of $ \mathcal{U}_\ast $ by $ \mathcal{U}_{\ast,\beta} $ in the definition of $ \hat{u} $, we end up with (note that here $ t_s = t_s(\beta) $)
\begin{equation}\label{pb-barenblatt-8} 
\begin{gathered}
\sup_{u_0 \in \LL^1(\Omega): \, \| u_0 \|_1 = \beta+\ell } \, \frac{\| u(t) \|_{\infty}}{\mathcal{S}(t,\| u_0 \|_1)}
\ge \frac{ K_0 }{K \left[ (1+\ell/\beta)^{\frac{2}{2+N(m_2-1)}} + \beta^{-\frac{2}{2+N(m_2-1)}} \, (\beta + \ell) \, t^{\frac{N}{2+N(m_2-1)}} \right]} \\
 \forall t \in \left( 0 ,  t_s(\beta) \wedge (\beta+\ell)^{\frac 2 N}  \right) .
\end{gathered}
\end{equation} 
In view of \eqref{pb-barenblatt-7} and \eqref{pb-barenblatt-8}, the infimum in \eqref{eq:def-sharpness} is indeed bounded from below by the same constant as in \eqref{pb-barenblatt-6}, and sharpness is finally established. 

We stress that the above construction applies to the Dirichlet problem as well (see Remark \ref{rem: dirichlet} below), since the subsolutions we exploit are compactly supported for short times. Moreover, we made no assumption at all on the domain, which means that for \emph{any} domain one cannot expect a better estimate than \eqref{pb-barenblatt-8}. This is coherent with the fact that smoothing estimates are actually \emph{equivalent} to Gagliardo-Nirenberg-Sobolev inequalities (we refer in particular to \cite[Theorem 5.3]{GM13} and \cite[Theorems 3.1, 3.3]{GM14}).      

\section{Well-posedness and basic properties}\label{sect: well}

In this section we deal with existence and uniqueness issues for solutions of problem \eqref{pb} (and related properties), so as to clarify what we mean by ``solution'' in the results stated above. Even though we are mainly interested in functions $ \phi \in C^1(\mathbb{R}) $ complying with \eqref{cond-phi-1}--\eqref{cond-phi-3}, in the sequel we shall allow for more general nonlinearities, that is we shall only make the following assumptions (see e.g.~\cite[Section 5.2]{Vbook}):
\begin{equation}\label{ass-gen-phi}
\phi : \, \mathbb{R} \mapsto \mathbb{R} \ \textrm{is continuous and strictly increasing, with } \lim_{u\to\pm\infty} \phi(u)=\pm\infty \, , \ \phi(0)=0 \, .
\end{equation}
Similar remarks hold for $ \Omega $: we only suppose that it is a general domain of $ \mathbb{R}^N $, regardless of boundedness, regularity or the validity of global functional inequalities like \eqref{NGN} and \eqref{dis-poin}. However, in the cases where the assumptions of Theorem \ref{thm01} are fulfilled, we can somewhat improve on the well-posedness theory outlined here (see Remark \ref{rem: uniq} below in this regard). 

%
Let us start off by providing an appropriate definition of weak solution.
\begin{den}[Weak solutions]\label{def-gmp}
Given $ u_0 \in \LL^1(\Omega) $, a measurable function $ u $ is a weak solution of the Neumann problem \eqref{pb} if, for all $ T>0 $,
\begin{equation*}\label{sol-p1}
u \in \LL^1(\Omega \times (0,T)) \, , \quad \phi(u) \in \LL^1_{\rm loc}(\Omega \times (0,T)) \, , \quad \nabla \phi(u) \in \LL^1\big((0,T);[\LL^2(\Omega)]^N\big)   
\end{equation*}
and 
\begin{equation*}\label{sol-p2}
\int_0^T \int_\Omega u(x,t) \, \eta_t(x,t) \, dx dt = -\int_\Omega u_0(x) \, \eta(x,0) \, dx + \int_0^T \int_\Omega \nabla \phi(u)(x,t) \cdot \nabla\eta(x,t) \, dx dt
\end{equation*}
for all $ \eta \in W^{1,\infty}((0,T);\LL^\infty(\Omega)) $ with $ \nabla{\eta} \in \LL^\infty((0,T);[\LL^2(\Omega)]^N) $, such that $ \eta(\cdot,T)\equiv 0 $.
\end{den} 

Actually, for initial data that are merely in $ \LL^1(\Omega) $, in general we cannot guarantee existence of a weak solution in the sense of Definition \ref{def-gmp}. Nevertheless, in such case there is still a natural way to define what one means by ``solution'', see Proposition \ref{thm-limit} below.

\begin{oss}[Other definitions of weak solution]\label{def-diff}\rm
As the reader may notice, our definition of weak solution slightly differs from \cite[Definition 11.2]{Vbook}. The point is that, as just recalled above, we aim at considering general domains $ \Omega $: in particular, density of functions that are regular up to the boundary need not hold e.g.~in $ H^1(\Omega) $. For this reason we do not assume $ \eta \in C^1(\overline{\Omega}\times [0,T]) $ but we ask that $ \eta $ belongs to a larger ``dual'' space, namely $ \nabla \eta \in [\LL^2(\Omega)]^N  $. As a reference for similar questions, see also \cite[Sections 2, 3]{GMP13}. 
\end{oss}

As a direct consequence of Definition \ref{def-gmp}, all weak solutions enjoy an important property.
\begin{pro}[Mass conservation]\label{prop-mass-c}
Let $u$ be any weak solution of \eqref{pb}. Then 
\begin{equation*}\label{cons-mass-th}
\int_\Omega u(x,t) \, dx = \int_\Omega u_0(x) \, dx \quad \textrm{for a.e. } t >0 \, .
\end{equation*}
\end{pro}
\begin{proof}
One can proceed exactly as in the proof of \cite[Proposition 9]{GMP13}, that is by using the (constant-in-space) test function $ \eta=\chi_{[0,t]} $ up to approximations. We stress that the finiteness of $ |\Omega| $ is not necessary here, since Definition \ref{def-gmp} allows one to pick (sufficiently regular) test functions that are independent of the space variable.
\end{proof}

Before stating a key existence result (i.e.~the analogue of \cite[Theorem 11.2]{Vbook}), we need to introduce the primitive function of $ \phi $, namely $ \psi(u):=\int_0^u \phi(v) \, dv $. Note that $ \psi(0)=0 $ and, by virtue of \eqref{ass-gen-phi}, $ \psi $ is $ C^1(\mathbb{R}) $, positive in $ \mathbb{R} \setminus \{ 0 \} $, strictly increasing in $ (0,+\infty) $ and strictly decreasing in $(-\infty,0)$, with $ \lim_{u\to\pm\infty} \psi(u)=+\infty $.

\begin{pro}[Existence and estimates]\label{thm-exi}
Let $ u_0 \in \LL^1(\Omega) $ with $ \psi(u_0) \in \LL^1(\Omega) $. Then there exists a weak solution $u$ of \eqref{pb}, which enjoys the following properties.
\begin{itemize}
\item \emph{Energy inequality:} $u$ satisfies
\begin{equation}\label{eest}
\int_0^T \int_{\Omega} \left| \nabla{\phi(u)}(x,t) \right|^2 dx dt + \int_\Omega \psi(u(x,T)) \, dx \le \int_\Omega \psi(u_0(x)) \, dx \quad \textrm{for a.e. } T > 0 
\end{equation}
and is referred to as a weak \emph{energy} solution.
\item \emph{Approximation:} $u$ is obtained as a limit of classical solutions of suitable non-degenerate parabolic problems.
\item \emph{$\LL^1$-contractivity and comparison:} if $v$ is another weak energy solution corresponding to some $ v_0 \in \LL^1(\Omega) $ with $ \psi(v_0) \in \LL^1(\Omega) $, then
\begin{equation}\label{L1-est}
\left\| u(t)-v(t) \right\|_1 \le \left\| u_0-v_0 \right\|_1 \quad \textrm{for a.e. } t > 0 \, .
\end{equation}
Moreover, if $ v_0 \le u_0 $ a.e.~in $ \Omega $ then $ u \le v $ a.e.~in $ \Omega \times \mathbb{R}^+ $.
\item \emph{Non-expansivity of the norms:} if in addition $ u_0 \in \LL^\infty(\Omega) $, there holds
\begin{equation}\label{eq: non-exp}
\left\| u(t) \right\|_p \le \left\| u_0 \right\|_p  \quad \forall p \in [1,\infty] \, , \ \textrm{for a.e. } t > 0 \, .
\end{equation}
\end{itemize}
\end{pro}
\begin{proof}
The procedure for constructing a weak energy solution is by now quite standard, though elaborate, so we prefer not to give full details here: for the reader's convenience, we refer to \cite[Sections 5.5, 11.2]{Vbook}, \cite[Section 3]{GMP13} and \cite[Section 1.3.2]{Mur} for accurate step-by-step proofs (in various contexts). 

We recall that the main idea is to begin by considering (regular) data $ u_0 \in \LL^1(\Omega)\cap \LL^\infty(\Omega) $, solve the problem upon replacing $ \phi $ with a suitable sequence of regular, non-degenerate nonlinearities $ \phi_n $ approximating $ \phi $ and then pass to the limit as $ n \to \infty $. At this level $ \Omega $ is supposed to be also bounded and regular. Afterwards, one first removes the assumptions on the domain by taking a sequence of regular domains such that $ \overline{\Omega}_n \Subset \Omega  $ which eventually covers the whole $ \Omega $, and finally removes the assumptions on the initial datum by picking a sequence of regular initial data $ u_{0n} \in \LL^1(\Omega)\cap \LL^\infty(\Omega) $ such that $ u_{0n} \to u_0 $ and $ \psi(u_{0n}) \to \psi(u_{0}) $ in $ \LL^1(\Omega)  $. In most of these steps estimates \eqref{eest}, \eqref{L1-est} and \eqref{eq: non-exp} (to be understood on the approximating sequences of solutions), plus suitable bounds on time derivatives (key to have compactness), are crucially exploited.

Let us just discuss a subtle question in the last passage to the limit, where we cannot use \eqref{eq: non-exp}. That is, in order to show that $ \nabla\phi(u_n) $ converges (weakly) to $ \nabla \phi(u) $ in $ \LL^2((0,T);[\LL^2(\Omega)]^N) $, it is essential to establish that $ \phi(u_n) $ converges to $ \phi(u) $ at least (weakly) in $ \LL^1_{\rm loc}(\Omega \times (0,T) ) $. To this aim, first of all note that by means of the local Poincar\'e inequality, for all bounded, regular $ \overline{\Omega}_0 \Subset \Omega $ there holds 
\begin{equation}\label{eq: poin-local}
\big\| \phi(u_n)(t) - \overline{\phi(u_n)(t)} \big\|_{\LL^2(\Omega_0)} \le C_P(\Omega_0) \left\| \nabla \phi(u_n)(t) \right\|_{\LL^2(\Omega_0)} \quad \textrm{for a.e. } t \in (0,T) \, , 
\end{equation} 
where for simplicity we still denote by $ \overline{f} $ the mean value of $ f \in \LL^1(\Omega_0) $ in $ \Omega_0 $. Thanks to the energy inequality we know, in particular, that 
\begin{equation}\label{sav-1}
\left| \{ x \in \Omega : \, \left| u_n(x,t) \right| > c \} \right| \le \frac{b}{\min\{\psi(c),\psi(-c)\}} \, , \quad b:= \limsup_{n\to\infty} \int_{\Omega} \psi(u_{0n}(x)) \, dx < \infty \, ,
\end{equation}
for all $c>0$, where we used the fact that $ \psi $ is increasing in $ (0,+\infty) $ and decreasing in $ (-\infty,0) $, with $ \psi(0)=0 $. Since $ \lim_{u \to \pm \infty} \psi(u) = + \infty $, upon picking $ c $ large enough (independently of $n$), from \eqref{sav-1} we infer that there exists $ E_n \subset \Omega_0 $ such that 
\begin{equation}\label{sav-2}
\left| u_n(x,t) \right| \le c \quad \textrm{for a.e. } x \in E_n  \, , \quad \left| E_n \right| \ge \frac{1}{2} \left| \Omega_0 \right| .
\end{equation}
As $ \phi $ is increasing, by combining \eqref{eq: poin-local} with \eqref{sav-2} we obtain
\begin{equation}\label{sav-3}
\big|  \overline{\phi(u_n)(t)}  \big| \le \sqrt{2}M + C_P(\Omega_0) \, \sqrt{{2}{\left|\Omega_0\right|^{-1}}} \left\| \nabla \phi(u_n)(t) \right\|_{\LL^2(\Omega_0)} , \quad M:=\max\{\phi(c),-\phi(-c)\} \, .
\end{equation}
By exploiting \eqref{sav-3} and again \eqref{eq: poin-local}, we therefore end up with 
\begin{equation}\label{sav-4}
\left\| \phi(u_n)(t) \right\|_{\LL^2(\Omega_0)} \le \sqrt{2\left| \Omega_0 \right|}\,M + \big(\sqrt2 + 1 \big) \, C_P(\Omega_0) \left\| \nabla \phi(u_n)(t) \right\|_{\LL^2(\Omega_0)} .
\end{equation}
Hence, if we square \eqref{sav-4}, integrate in $ (0,T) $ and use the energy inequality, we finally deduce that $ \phi(u_n) $ is bounded in $ \LL^2(\Omega_0 \times [0,T]) $, which is enough for our purposes (pointwise convergence to $ \phi(u) $ is already ensured e.g.~by the $\LL^1$-contractivity). 

We stress that in the plain porous-medium or fast-diffusion cases, namely ${\phi(u)=u^m}$ for some $m>0$, the above convergence is a simple consequence of the boundedness of $u_n$ in $ \LL^{m+1}(\Omega\times(0,T)) $.
\end{proof}

\begin{pro}[Limit solutions]\label{thm-limit}
There exists a well-defined, $1$-Lipschitz map acting from $ \LL^1(\Omega) $ to $ \LL^\infty(\mathbb{R}^+;\LL^1(\Omega)) $ which associates with each $ u_0 \in \LL^1(\Omega) $ the limit $u $ in $ \LL^\infty(\mathbb{R}^+;\LL^1(\Omega)) $ of (any) sequence of weak energy solutions $ u_n $ corresponding to initial data $ u_{0n} \in \LL^1(\Omega) $, with $ \psi(u_{0n}) \in \LL^1(\Omega) $, such that $ u_{0n} \to u_0 $ in $ \LL^1(\Omega) $. The $ \LL^1 $-contraction and the comparison properties are preserved at the limit.
\end{pro}
\begin{proof}
It is a standard fact, see e.g.~\cite[Theorems 6.2 and 11.3]{Vbook} for detailed proofs: one exploits \eqref{L1-est} to show that $ u_n $ is a Cauchy sequence in $ \LL^\infty(\mathbb{R}^+;\LL^1(\Omega)) $, which therefore admits a limit $ u $. Given another $ v_0 \in \LL^1(\Omega) $, by taking any of the corresponding approximating sequences $  v_{0n} $, considering the associated sequence of solutions $ v_n $ as above and passing to the limit in \eqref{L1-est}, we establish at once that the map is well defined (i.e.~the limit point is independent of the approximating sequence) and $1$-Lipschitz. Alternatively, this is nothing but an application of the bounded extension Theorem for Lipschitz maps between Banach spaces. Comparison is then preserved in view e.g.~of pointwise convergence (up to subsequences). 
\end{proof}

Uniqueness of weak (energy) solutions for problem \eqref{pb} is due to a classical trick, which goes back to Ole\u{\i}nik \cite{OL2}. However, we emphasize that in the forthcoming theorem we do not require further integrability or boundedness properties on $ u $ or $ \phi(u) $, in contrast with known results available in the literature: see for instance \cite[Theorems 5.3 and 11.1]{Vbook} and \cite[Propositions 6 and 8]{GMP13} in the local case or \cite[Theorem 1.1]{DQR} and \cite[Theorem 2.4 and Corollaries 2.7--2.9]{DEJ} in the case of generalized nonlocal filtration equations with rough kernels. This can be done by means of an elementary truncation argument.
\begin{pro}[Ole\u{\i}nik's uniqueness]\label{thm-uniq}
There exists at most one weak solution $u$ of \eqref{pb} such that 
\begin{equation}\label{cond: olei}
\nabla \phi(u)  \in \LL^2(\Omega \times (0,T)) \quad  \forall T>0 \, .
\end{equation}
\end{pro}
\begin{proof}
The basic idea consists in plugging
$ \eta(x,t)=\int_t^T \left[ \phi(u(x,s)) - \phi(v(x,s)) \right] ds $
in the weak formulation satisfied by the difference between $u$ and $v$ (the latter being any other solution fulfilling \eqref{cond: olei}) and perform the same computations as in the proof of \cite[Theorem 5.3]{Vbook} (see also \cite[Propositions 6 and 8]{GMP13}). We only need to justify the use of $ \eta $ as a test function. To this aim, let us set 
$ 
\eta_n := \int_t^T \left[ \phi(u_n(x,s)) - \phi(v_n(x,s)) \right] ds \, , 
$
where $ u_n:= -n \vee (n \wedge u) $ and $ v_n:= -n \vee (n \wedge v) $.
Upon picking $ \eta_n $ in place of $ \eta $ we obtain the key identity
\begin{equation*}\label{keyid}
\begin{aligned} 
\int_0^T \int_\Omega \left( u(x,t)-v(x,t) \right) \left[ \phi(u_n(x,t)) - \phi(v_n(x,t)) \right] dx dt &  \\
+ \int_0^T \int_\Omega \nabla \!\left[ \phi(u)-\phi(v) \right]\!(x,t) \cdot \left( \int_t^T \nabla\!\left[ \phi(u_n) - \phi(v_n) \right]\!(x,s) \, ds  \right)  dx dt & =0  \, .
\end{aligned}
\end{equation*}
Since $ \phi $ is increasing we can pass to the limit by monotone convergence in the first integral, while in the second integral we can exploit dominated convergence in view of \eqref{cond: olei}. The identity $ u \equiv v $ then follows as in the proof of \cite[Theorem 5.3]{Vbook}.
\end{proof}

\begin{oss}[The concept of solution and continuity properties]\label{rem: uniq}\rm
When referring to the ``solution'' of problem \eqref{pb} in Theorems \ref{thm01}, \ref{thm-asym-zero} and \ref{thm-asym-nonzero}, we shall mean the weak energy solution constructed in Proposition \ref{thm-exi} if $ u_0 \in \LL^1(\Omega) $ with $ \psi(u_0) \in \LL^1(\Omega) $ (note that initial data belonging to $ \LL^1(\Omega)\cap\LL^\infty(\Omega) $ are always included in this category), which by virtue of Proposition \ref{thm-uniq} is the \emph{unique} weak solution complying with \eqref{cond: olei}. In the general case $ u_0 \in \LL^1(\Omega) $ we shall just mean the limit solution provided by Proposition \ref{thm-limit}. We stress that, in agreement with the latter results, such solutions can always be thought as limits (after several approximations) of regular solutions to suitable non-degenerate parabolic problems, a fact that we shall exploit in order to justify some of the computations performed in Sections \ref{sect: short} and \ref{sect: long}. 

Weak energy solutions are in fact continuous curves in $ \LL^1(\Omega) $, i.e.~$ u \in C([0,\infty);\LL^1(\Omega)) $. This is due to an alternative method for constructing such solutions by means of \emph{time discretization}, which takes advantage of the celebrated Crandall-Liggett theorem: see \cite[Chapter 10]{Vbook}. The $ \LL^1 $-continuity is then trivially inherited by limit solutions: for that reason, and in order to lighten the reading, throughout Sections \ref{sect: res}, \ref{sect: short}, \ref{sect: long} we wrote ``$ \forall t $'' in place of ``for a.e.~$t$''. 

Actually, by exploiting the latter properties, we can infer uniqueness in a wider class of solutions, namely that of functions which are weak energy solutions for all \emph{positive} times (i.e.~with the time origin shifted to $ \varepsilon $, for all $ \varepsilon>0 $) and belong to $ C([0,\infty);\LL^1(\Omega)) $. The corresponding argument is analogous to the one used in the proof of \cite[Theorem 6.12]{Vbook}. We stress that under assumptions \eqref{cond-phi-1}--\eqref{cond-phi-3}, thanks to Theorem \ref{thm01}, \emph{a posteriori} limit solutions are included in such class.
\end{oss}
%
%
%
%
%

\section{Short-time estimates: proofs}\label{sect: short}

As pointed out above, an improvement of the arguments of \cite{GM13} allows us to remove the hypothesis that $ |\Omega| $ is finite, which is relevant also in the single-power framework. The technical novelty, basically contained in Lemma \ref{lemma: q_0-t-ast} below, consists in the use of Young-type inequalities rather than interpolation inequalities, which prove to be better suited to handle typical recurrence relations that arise out of Moser iterations (in particular additional terms due to the fact that we work with Neumann problems).

We begin with the analysis of the case $ m_1>m_2 $, where proofs are more involved since it is not possible to deduce from \eqref{cond-phi-2}--\eqref{cond-phi-3} a global lower bound on $\phi'(u)$ relying on a \emph{single} power. 

\begin{lem}\label{lemma: q_0-t-ast}
Let the hypotheses of Theorem \ref{thm01} be fulfilled, with the additional assumptions $ m_1>m_2 $, $ u_0 \in \LL^\infty(\Omega) $ and $ q_0>1 $. Let 
\begin{equation}\label{def-t-ast}
t^*:=\sup\left\{t \ge 0 : \ \Vert u(t) \Vert_{\infty} > 1 \right\}
\end{equation}
and suppose that $ t^\ast>0 $. Then there holds the smoothing estimate
\begin{equation}\label{case-1-smooth-m2}
\left\| u(t) \right\|_{\infty} \leq K \left( t^{-\frac{N}{2q_0+N(m_2-1)}} \, \left\Vert u_0 \right\Vert_{q_0}^{\frac{2q_0}{2q_0+N(m_2-1)}} + \left\Vert u_0 \right\Vert_{q_0} \right) \quad \forall t \in \left( 0 , t^\ast \right) ,
\end{equation}
where $K$ is a positive constant depending only on $ m_1,m_2,c_1,c_2,C_S,N,q_0 $.
\end{lem}
\begin{proof}
Let $ M:=\Vert u_0\Vert_{\infty} $. Note that the non-expansivity of the norms (Proposition \ref{thm-exi}) and the assumption $ t^\ast>0 $ imply $ \| u_0 \|_\infty>1 $. Still by means of the non-expansivity of the norms, we know that $\Vert u(t)\Vert_{\infty} \leq M $ for all $ t>0 $, so that $u$ takes on values in the interval $[-M,M]$. As a consequence, the bounds \eqref{cond-phi-2}--\eqref{cond-phi-3} ensure that
\begin{equation}\label{upperboundphiprime}
\frac{c}{M^{m_1-m_2}} \, |u(x,t)|^{m_1-1} \le \phi'(u(x,t)) \quad  \textrm{for a.e. } (x,t) \in \mathbb{R}^d \times \mathbb{R}^+ \, ,
\end{equation}
where $ c:= c_1 \wedge c_2 $. Thanks to \eqref{upperboundphiprime}, which allows us to partially resort to the single-power case (\emph{i.e.}~to the standard porous medium equation), first of all we can set up a Moser-iteration scheme in the same spirit as \cite{GM13,GM14}, and then by means of another iteration we shall remove the dependence on $M$. 

Given $t>0$, let us consider the sequence of time steps $t_k=(1-2^{-k})\,t$, for all $ k \in \mathbb{N} $. Clearly, $t_0=0$ and $t_{\infty}=t$. Also, let $ p_k $ be an increasing sequence of positive numbers such that $p_0=q_0$ and $p_{\infty}=\infty$, which we shall explicitly define later. By multiplying the differential equation in \eqref{pb} by $u^{p_k-1}$, integrating by parts in $\Omega\times(t_k,t_{k+1})$ and using \eqref{upperboundphiprime}, we obtain: 
\begin{equation}\label{la3.2}
\begin{aligned}
& \frac{4 \, c \, p_k \,(p_k-1)}{M^{m_1-m_2}\,(m_1+p_k-1)^2} \, \int_{t_k}^{t_{k+1}} \int_{\Omega} \left| \nabla \! \left(u^{\frac{m_1+p_k-1}{2}}\right)\!(x,t)\right|^2 dx dt \\
\le & p_k \, (p_k-1) \, \int_{t_k}^{t_{k+1}} \int_{\Omega} \left|u(x,t)\right|^{p_k-2} \phi^\prime(u(x,t)) \left| \nabla u(x,t)\right|^2 dx dt \\
= & \Vert u(t_k)\Vert_{p_k}^{p_k} - \Vert u(t_{k+1})\Vert_{p_k}^{p_k} \\
\le & \Vert u(t_k)\Vert_{p_k}^{p_k} \, .
\end{aligned}
\end{equation}
We point out that \emph{a priori} $ u $ may not possess enough regularity to justify rigorously the above computation: nevertheless, this issue can be overcome by means of suitable approximation schemes, see e.g.~\cite[Proof of Lemma 3.3 and Remark 2]{GMP13} and the first part of Remark \ref{rem: uniq} (similar comments apply to the computations carried out below). In order to handle the left-hand side of inequality \eqref{la3.2}, it is convenient to apply the Gagliardo-Nirenberg-Sobolev inequality in the form \eqref{NGN-bis} to the function $ f=u^{(m_1+p_k-1)/{2}} $:
\begin{equation}\label{la3.2-bis}
\begin{aligned}
& \frac{2 \, c \, p_k \,(p_k-1)}{M^{m_1-m_2}\, C_S^{\frac{2}{\vartheta}} \,(m_1+p_k-1)^2} \, \int_{t_k}^{t_{k+1}} \frac{\left\| u(t) \right\|_{\frac{r(m_1+p_k-1)}{2}}^{\frac{m_1+p_k-1}{\vartheta}}}{\left\| u(t) \right\|^{\frac{(1-\vartheta)(m_1+p_k-1)}{\vartheta}}_{\frac{s(m_1+p_k-1)}{2}} } \, dt \\
\leq & \Vert u(t_k)\Vert_{p_k}^{p_k} + \frac{4 \, c \, p_k \,(p_k-1)}{M^{m_1-m_2}\,(m_1+p_k-1)^2} \, \int_{t_k}^{t_{k+1}} \left\| u(t) \right\|^{m_1+p_k-1}_{\frac{s(m_1+p_k-1)}{2}} dt \, .
\end{aligned}
\end{equation}
By choosing 
$$ s=\frac{2p_k}{m_1+p_k-1} \quad \textrm{and} \quad r=2+\frac{2s}{N}=2\,\frac{(N+2)p_k+N(m_1-1)}{N(m_1+p_k-1)} $$ 
(note that $ s<2 $ and \eqref{NGN-parameters-r-s-1}--\eqref{NGN-parameters-r-s-2} are always fulfilled), recalling \eqref{NGN-parameters-theta}, the definition of $ t_k $ and using the non-expansivity of the norms, from \eqref{la3.2-bis} we infer
\begin{equation}\label{la3.2-ter}
\begin{aligned}
& \frac{c \, p_k \,(p_k-1)\, t}{M^{m_1-m_2} \, C_S^{\frac{2}{\vartheta}} \,(m_1+p_k-1)^2 \, 2^k} \times \frac{\left\| u(t_{k+1}) \right\|_{p_{k+1}}^{p_{k+1}}}{\left\| u(t_k) \right\|^{\frac{2p_k}{N}}_{p_k} } \\
\leq & \Vert u(t_k)\Vert_{p_k}^{p_k} + \frac{2 \, c \, p_k \,(p_k-1)\,t}{M^{m_1-m_2}\,(m_1+p_k-1)^2\,2^k} \, \left\| u(t_k) \right\|^{m_1+p_k-1}_{p_k} ,
\end{aligned}
\end{equation}
where $  p_{k+1} $ is defined recursively by
\begin{equation*}\label{eq: pk-pk+1}
p_{k+1}=\frac{N+2}{N} \, p_k+m_1-1 \, ,
\end{equation*}	
or equivalently
\begin{equation}\label{eq: pk-explicit}
p_{k}= \left[ q_0 + \frac{N(m_1-1)}{2} \right] \left( \frac{N+2}{2} \right)^k - \frac{N(m_1-1)}{2} \, .
\end{equation}	
From here on we shall denote by $D$ a generic positive constant that depends only on $ m_1,m_2,c,C_S,N,q_0 $. Hence, upon observing that $ C_S^{1/\vartheta} $ can be bounded independently of $ s $ and $r$ chosen as above, estimate \eqref{la3.2-ter} reads
\begin{equation}\label{la3.2-quater}
\left\| u(t_{k+1}) \right\|_{p_{k+1}}^{p_{k+1}} \le D \left( \frac{2^k \, M^{m_1-m_2}}{t} \, \Vert u(t_k)\Vert_{p_k}^{\frac{N+2}{N}\,p_k} + \Vert u(t_k)\Vert_{p_k}^{\frac{N+2}{N}\,p_k + m_1-1} \right) .
\end{equation} 
By combining the non-expansivity of the norms, the monotonicity of $ p_k $, interpolation and Young's inequality, we obtain:
$$ \| u(t_k) \|_{p_k} \le \| u_0 \|_\infty + \| u_0 \|_{q_0} = M + \| u_0 \|_{q_0} \, , $$
so that 	\eqref{la3.2-quater} implies 
\begin{equation}\label{la3.2-young}
\left\| u(t_{k+1}) \right\|_{p_{k+1}} \le D^{\frac{k+1}{p_{k+1}}} \left( \frac{M^{m_1-m_2}}{t} \,  + M^{m_1-1} + \| u_0 \|_{q_0}^{m_1-1} \right)^{\frac{1}{p_{k+1}}} \Vert u(t_k)\Vert_{p_k}^{\frac{N+2}{N}\,\frac{p_k}{p_{k+1}}} .
\end{equation} 
The iteration of \eqref{la3.2-young} and again the non-expansivity of the norms yield
\begin{equation*}\label{la3.2-young-iter}
\begin{aligned}
\left\| u(t) \right\|_{p_{k+1}} 
\le  D^{\frac{\sum_{h=1}^{k+1} h \left( \frac{N+2}{N} \right)^{k+1-h} }{p_{k+1}}} \left( \frac{M^{m_1-m_2}}{t} \,  + M^{m_1-1} + \| u_0 \|_{q_0}^{m_1-1} \right)^{\frac{\sum_{h=0}^{k} \left( \frac{N+2}{N} \right)^h }{p_{k+1}}} \Vert u_0 \Vert_{q_0}^{\left(\frac{N+2}{N}\right)^{k+1} \frac{q_0}{p_{k+1}}} .
\end{aligned}
\end{equation*} 
By letting $ k \to \infty $, in view of \eqref{eq: pk-explicit} we end up with
\begin{equation*}\label{la3.2-young-limit}
\left\| u(t) \right\|_{\infty} \le D \left( \frac{M^{m_1-m_2}}{t} \,  + M^{m_1-1} + \| u_0 \|_{q_0}^{m_1-1} \right)^{\frac{N}{2q_0+N(m_1-1)}} \left\| u_0 \right\|_{q_0}^{\frac{2q_0}{2q_0+N(m_1-1)}} ,
\end{equation*}
whence
\begin{equation}\label{la3.2-young-limit-bis}
\left\| u(t) \right\|_{\infty} \le D \left( \frac{M^{\frac{m_1-m_2}{m_1-1}}}{t^{\frac{1}{m_1-1}}} \,  + M + \| u_0 \|_{q_0} \right)^{\frac{N(m_1-1)}{2q_0+N(m_1-1)}} \left\| u_0 \right\|_{q_0}^{\frac{2q_0}{2q_0+N(m_1-1)}} .
\end{equation}
%
Thanks to Young's inequality  
\begin{equation*}\label{young-first}
A^{\theta} \, B^{1-\theta} \le \varepsilon \, \theta \, A + \varepsilon^{-\frac{\theta}{1-\theta}} \, (1-\theta) \, B  \quad \forall A,B,\varepsilon > 0 \, , \quad \theta := \frac{N(m_1-1)}{2q_0+N(m_1-1)} \, ,
\end{equation*} 
from \eqref{la3.2-young-limit-bis} we infer 
\begin{equation}\label{est-iter-1}
\left\| u(t) \right\|_{\infty} \le D \left[ \frac{M^{\frac{m_1-m_2}{m_1-1} \, \theta }}{t^{\frac{\theta}{m_1-1}}} \left\| u_0 \right\|_{q_0}^{1-\theta} + \varepsilon \, \theta \, M + \left[ \varepsilon \, \theta + \varepsilon^{-\frac{\theta}{1-\theta}} \, (1-\theta) \right] \| u_0 \|_{q_0} \right] . 
\end{equation}
Now let us pick 
$$  \varepsilon= \left(D \, \theta \, 2^{\frac{1+\theta}{m_1-1}} \right)^{-1}  , $$ 
so that \eqref{est-iter-1} can be rewritten as 
\begin{equation}\label{est-iter-2}
\left\| u(t) \right\|_{\infty} \le \frac{1}{2^{\frac{1+\theta}{m_1-1}}} \, M + D \, \frac{M^{\frac{m_1-m_2}{m_1-1}\,\theta}}{t^{\frac{\theta}{m_1-1}}} \, \| u_0 \|_{q_0}^{1-\theta}  + D \, \| u_0 \|_{q_0} 
\end{equation}
for another positive constant $D$ which, as usual, we do not relabel. Our goal is to (partially) remove the dependence of the right-hand side of \eqref{est-iter-2} on $M$: to this end, first of all we exploit a classical $ t/2 $-shift argument (see e.g.~\cite[Proof of Theorem 3.2]{GM13}) combined with the non-expansivity of the norms, which yields
\begin{equation}\label{est-iter-3}
\big\| u\big(t/2^k\big) \big\|_{\infty} \le \frac{\big\| u\big(t/2^{k+1}\big) \big\|_{\infty}}{2^{\frac{1+\theta}{m_1-1}}} + 2^{\frac{\theta(k+1)}{m_1-1}} \, D \, \frac{M^{\frac{m_1-m_2}{m_1-1}\,\theta}}{t^{\frac{\theta}{m_1-1}}} \, \| u_0 \|_{q_0}^{1-\theta} + D \, \| u_0 \|_{q_0} \quad \forall k \in \mathbb{N} \, .
\end{equation}
By iterating \eqref{est-iter-3} we therefore obtain 
\begin{equation*}\label{est-iter-4}
\left\| u(t) \right\|_{\infty} \le \frac{M}{2^{\frac{(1+\theta)(\ell+1)}{m_1-1}}} + 2^{\frac{\theta}{m_1-1}} \, D \, \frac{M^{\frac{m_1-m_2}{m_1-1}\,\theta}}{t^{\frac{\theta}{m_1-1}}}  \, \| u_0 \|_{q_0}^{1-\theta} \, \sum_{k=0}^{\ell} 2^{-\frac{k}{m_1-1}} + D \, \| u_0 \|_{q_0} \, \sum_{k=0}^\ell 2^{-\frac{(1+\theta) k}{m_1-1}} 
\end{equation*}  
for all $ \ell \in \mathbb{N} $, whence, upon taking limits as $ \ell \to \infty $, 
\begin{equation}\label{est-iter-limit}
\| u(t) \|_{\infty} \le  D \left( \frac{M^{\frac{m_1-m_2}{m_1-1}\,\theta}}{t^{\frac{\theta}{m_1-1}}}  \, \| u_0 \|_{q_0}^{1-\theta} + \| u_0 \|_{q_0} \right) .
\end{equation} 
In order to remove definitively the dependence of the right-hand side of \eqref{est-iter-limit} on $ M $ we can argue in a similar way as above to get, by means of Young's inequality,
\begin{equation}\label{est-iter-last-1}
\| u(t) \|_{\infty} \le  D \left( \varepsilon \, \theta_\star \, M + \varepsilon^{-\frac{\theta_\star}{1-\theta_\star}} \, (1-\theta_\star) \, \frac{\| u_0 \|_{q_0}^{\frac{2q_0}{2q_0+N(m_2-1)}}}{t^{\frac{N}{2q_0+N(m_2-1)}}}  + \| u_0 \|_{q_0} \right) ,
\end{equation} 
where we set
$$ \theta_\star := \tfrac{m_1-m_2}{m_1-1} \, \theta = \frac{N(m_1-m_2)}{2q_0+N(m_1-1)} \, . $$
By choosing
$$ \varepsilon=\left(D \, \theta_\star \, 2^{\frac{2q_0+2N (m_2-1)}{(m_2-1)[2q_0+N(m_2-1)]}}\right)^{-1} $$ 
and exploiting again a $ t/2 $-shift argument, from \eqref{est-iter-last-1} we infer 
\begin{equation}\label{est-iter-last-2}
\big\| u\big(t/2^k\big) \big\|_{\infty} \le \frac{\big\| u\big(t/2^{k+1}\big) \big\|_{\infty}}{2^{\frac{2q_0+2N (m_2-1)}{(m_2-1)[2q_0+N(m_2-1)]}}} + 2^{\frac{N(k+1)}{2q_0+N(m_2-1)}} \, D \, \frac{\| u_0 \|_{q_0}^{\frac{2q_0}{2q_0+N(m_2-1)}}}{t^{\frac{N}{2q_0+N(m_2-1)}}} +  D \, \| u_0 \|_{q_0} \quad \forall k \in \mathbb{N} \, .
\end{equation} 	
It is apparent that \eqref{est-iter-last-2} is of the same type as \eqref{est-iter-3}: by carrying out an analogous iteration, we therefore end up with \eqref{case-1-smooth-m2}.
\end{proof}	

We now extend the above result to $ q_0=1 $.
\begin{cor}\label{lemma: q_0-t-ast-ext}
Let the hypotheses of Theorem \ref{thm01} be fulfilled, with the additional assumptions $ m_1>m_2 $ and $ u_0 \in \LL^\infty(\Omega) $. Let $ t^\ast >0 $ be defined by \eqref{def-t-ast}. Then there holds the smoothing estimate
\begin{equation}\label{case-1-smooth-m2-ext}
\left\| u(t) \right\|_{\infty} \leq K \left( t^{-\frac{N}{2q_0+N(m_2-1)}} \left\| u_0 \right\|_{q_0}^{\frac{2q_0}{2q_0+N(m_2-1)}} + \left\| u_0 \right\|_{q_0} \right) \quad \forall t \in \left( 0 , t^\ast \right) ,
\end{equation}
where $K$ is a positive constant depending only on $ m_1,m_2,c_1,c_2,C_S,N $ (and in particular independent of $  q_0 \ge 1 $).
\end{cor}
\begin{proof}    
First of all, let us show that \eqref{case-1-smooth-m2} holds down to $ q_0=1 $: note that we cannot simply let $ q_0 \to 1^+ $ since the multiplicative constant blows up (this is due to \eqref{la3.2} evaluated at $ p_k=q_0=1 $). Hence, to begin with, fix any $ q_0>1 $ and plug the interpolation inequality $ \|u_0 \|_{q_0} \le \| u_0 \|_{\infty^{\phantom{a}}}^{1-1/q_0} \, \| u_0 \|_1^{1/q_0} $ in \eqref{case-1-smooth-m2} to get 
\begin{equation}\label{case-1-smooth-m2-ext-proof-1}
\begin{aligned}
\left\| u(t) \right\|_{\infty} \leq K \left\| u_0 \right\|_{\infty^{\phantom{a}}}^{\frac{2(q_0-1)}{2q_0+N(m_2-1)}} \left( t^{-\frac{N}{2q_0+N(m_2-1)}} \, \Vert u_0 \Vert_{1}^{\frac{2}{2q_0+N(m_2-1)}} + \Vert u_0 \Vert_{\infty^{\phantom{a}}}^{\frac{N(m_2-1)(q_0-1)}{q_0[2q_0+N(m_2-1)]}} \, \Vert u_0 \Vert_{1}^{\frac{1}{q_0}} \right) .
\end{aligned}
\end{equation}
By means of the usual $t/2$-shift argument and the non-expansivity of the norms, from \eqref{case-1-smooth-m2-ext-proof-1} we deduce
\begin{equation}\label{case-1-smooth-m2-ext-proof-2}
\begin{aligned}
 \big\| u\big(t/2^k\big) \big\|_{\infty} \leq & 2^{\frac{N(k+1)}{2q_0+N(m_2-1)}} \, K \, \big\| u\big(t/2^{k+1}\big) \big\|_{\infty}^{\frac{2(q_0-1)}{2q_0+N(m_2-1)}}  \\
& \times \left( t^{-\frac{N}{2q_0+N(m_2-1)}} \, \Vert u_0 \Vert_{1}^{\frac{2}{2q_0+N(m_2-1)}} + \left\| u_0 \right\|_{\infty^{\phantom{a}}}^{\frac{N(m_2-1)(q_0-1)}{q_0[2q_0+N(m_2-1)]}} \, \Vert u_0 \Vert_{1}^{\frac{1}{q_0}} \right)
\end{aligned}
\end{equation}
for all $  k \in \mathbb{N} $. A straightforward iteration of \eqref{case-1-smooth-m2-ext-proof-2} yields 
\begin{equation}\label{case-1-smooth-m2-ext-proof-4}
\Vert u(t) \Vert_{\infty} \leq D \left( t^{-\frac{N}{2+N(m_2-1)}} \, \Vert u_0 \Vert_{1}^{\frac{2}{2+N(m_2-1)}} + \Vert u_0 \Vert_{\infty^{\phantom{a}}}^{\frac{N(m_2-1)(q_0-1)}{q_0[2+N(m_2-1)]}} \, \Vert u_0 \Vert_{1}^{\frac{2q_0+N(m_2-1)}{q_0[2+N(m_2-1)]}} \right) ,
\end{equation}
where $ D $ denotes again a generic positive constant (that we do not relabel below) depending only on the quantities $m_1,m_2,c_1,c_2,C_S,N,q_0 $. In order to remove $\| u_0 \|_{\infty} $ from the right-hand side of \eqref{case-1-smooth-m2-ext-proof-4}, one can proceed by combining Young's inequality with a $t/2$-shift argument similarly to the proof of Lemma \ref{lemma: q_0-t-ast}, so as to obtain
\begin{equation*}\label{case-1-smooth-m2-ext-proof-6}
\Vert u(t) \Vert_{\infty} \leq D \left( t^{-\frac{N}{2+N(m_2-1)}} \, \Vert u_0 \Vert_{1}^{\frac{2}{2+N(m_2-1)}} + \Vert u_0 \Vert_{1} \right) ,
\end{equation*} 
which is precisely \eqref{case-1-smooth-m2} in the case $ q_0=1 $. 

Finally, we are left with proving \eqref{case-1-smooth-m2-ext}, namely \eqref{case-1-smooth-m2} with a constant $ K $ independent of $q_0$. In fact it is straightforward to verify that all of the estimates provided along the proof of Lemma \ref{lemma: q_0-t-ast} depend continuously on $  q_0 \in (1,\infty) $ and are stable as $ q_0 \to \infty $. On the other hand, the argument we used here to extend the estimate to $ q_0=1 $ can be performed analogously for any $ q_0^\prime \in (1,q_0) $, thus providing the same estimate as \eqref{case-1-smooth-m2} with a constant $ K $ that stays bounded as $ q_0^\prime \to 1 $. 
\end{proof}  

In order to obtain estimate \eqref{thm01-smooth-est} and therefore complete the proof of Theorem \ref{thm01} in the case $ m_1>m_2 $, we need to bound $ t^\ast $ by quantities that only depend on the initial datum. 
   
\begin{proof}[Proof of Theorem \ref{thm01} ($ m_1>m_2 $)]   
Let $u_0 \in \LL^1(\Omega) \cap \LL^\infty(\Omega)$. In the case $ t^\ast=\infty $, from Corollary \ref{lemma: q_0-t-ast-ext} we deduce the validity of 
\begin{equation}\label{eq: est-thm01-1}
\Vert u(t) \Vert_{\infty} \leq K \left( t^{-\frac{N}{2q_0+N(m_2-1)}} \, \Vert u_0 \Vert_{q_0}^{\frac{2q_0}{2q_0+N(m_2-1)}} + \Vert u_0 \Vert_{q_0} \right) \quad \forall t > 0 \, .  
\end{equation}
Since $ m_1>m_2 $, it is straightforward to verify that \eqref{eq: est-thm01-1} implies \eqref{thm01-smooth-est}. 

Suppose now that $ t^\ast=0 $. This means that, by the definition of $ t^\ast $, we can assume with no loss of generality that $ \phi(u) $ is of porous medium type with $ m=m_1 $. In other words, we are allowed to replace \eqref{cond-phi-2}--\eqref{cond-phi-3} with
\begin{equation}  \label{cond-phi-23-mod}
c_1 \, |u|^{m_1-1}  \le \phi^\prime(u) \quad \forall u \in \mathbb{R} \, .
\end{equation}
One can then proceed exactly as in the proofs of Lemma \ref{lemma: q_0-t-ast} and Corollary \ref{lemma: q_0-t-ast-ext} (which in fact become simpler under \eqref{cond-phi-23-mod}) to get 
\begin{equation}\label{eq: est-thm01-2}
\Vert u(t) \Vert_{\infty} \leq K \left( t^{-\frac{N}{2q_0+N(m_1-1)}} \, \Vert u_0 \Vert_{q_0}^{\frac{2q_0}{2q_0+N(m_1-1)}} + \Vert u_0 \Vert_{q_0} \right) \quad \forall t > 0 \, .  
\end{equation} 
Again, the condition $ m_1>m_2 $ ensures that from \eqref{eq: est-thm01-2} there follows \eqref{thm01-smooth-est}. 

We are therefore left with the case $ t^\ast \in (0,\infty) $. First of all note that, by taking $ t^\ast $ as the new time origin and reasoning as in the case $ t^\ast=0 $ (by also using the non-expansivity of the norms), we obtain:
\begin{equation}\label{eq: est-thm01-tstar}
\Vert u(t) \Vert_{\infty} \leq K \left[ \left( t-t^\ast \right)^{-\frac{N}{2q_0+N(m_1-1)}} \, \Vert u_0 \Vert_{q_0}^{\frac{2q_0}{2q_0+N(m_1-1)}} + \Vert u_0 \Vert_{q_0} \right] \quad \forall t > t^\ast \, .  
\end{equation} 
Hence, by virtue of \eqref{case-1-smooth-m2-ext} and \eqref{eq: est-thm01-tstar}, we end up with
\begin{equation}\label{eq: est-thm01-tstar-do}
	\Vert u(t)\Vert_{\infty}\leq 
	\begin{cases}
	K \left( t^{-\frac{N}{2q_0+N(m_2-1)}} \, \Vert u_0\Vert_{q_0}^{\frac{2q_0}{2q_0+N(m_2-1)}} + \Vert u_0\Vert_{q_0} \right) & \forall t \in\left( 0 , t^\ast \right] , \\
	K \left[ \left( t-t^\ast \right)^{-\frac{N}{2q_0+N(m_1-1)}} \, \Vert u_0\Vert_{q_0}^{\frac{2q_0}{2q_0+N(m_1-1)}} + \Vert u_0\Vert_{q_0} \right] & \forall t > t^\ast \, .
\end{cases}
\end{equation}
By the definition of $ t^\ast $, and exploiting \eqref{eq: est-thm01-tstar-do} evaluated at $ t=t^\ast $, we can infer the inequality  
\begin{equation}\label{eq: est-thm01-tstar-est}
1 \le K \left( {t^\ast}^{-\frac{N}{2q_0+N(m_2-1)}} \Vert u_0\Vert_{q_0}^{\frac{2q_0}{2q_0+N(m_2-1)}} + \Vert u_0\Vert_{q_0} \right) .
\end{equation}
Let us first suppose 
\begin{equation}\label{eq: est-thm01-tstar-q0}
\| u_0 \|_{q_0} \le \frac{1}{2K} \, , 
\end{equation}
so that \eqref{eq: est-thm01-tstar-est} yields 
\begin{equation}\label{eq: est-thm01-tstar-est-1}
t^\ast \le K^\prime \left\| u_0 \right\|_{q_0}^{\frac{2q_0}{N}} \, , \quad K^\prime := \left( 2K \right)^{\frac{2q_0+N(m_2-1)}{N}} \, .
\end{equation}   
On the other hand, the lower branch of \eqref{eq: est-thm01-tstar-do} implies the validity of 
\begin{equation*}
\Vert u(t)\Vert_{\infty} \leq \underbrace{2^{\frac{N}{2q_0+N(m_1-1)}} \, K}_{K^{\prime\prime}} \left( t^{-\frac{N}{2q_0+N(m_1-1)}} \, \Vert u_0 \Vert_{q_0}^{\frac{2q_0}{2q_0+N(m_1-1)}} + \Vert u_0\Vert_{q_0} \right) \quad \forall t \geq 2t^\ast \, ,
\end{equation*}   
which in turn implies
\begin{equation}\label{eq: est-thm01-tstar-est-2}
\Vert u(t)\Vert_{\infty} \leq K^{\prime\prime} \left( t^{-\frac{N}{2q_0+N(m_1-1)}} \, \Vert u_0 \Vert_{q_0}^{\frac{2q_0}{2q_0+N(m_1-1)}} + \Vert u_0\Vert_{q_0} \right) \quad \forall t \geq 2 K^\prime \left\| u_0 \right\|_{q_0}^{\frac{2q_0}{N}}
\end{equation}   
in view of \eqref{eq: est-thm01-tstar-est-1}. So, by exploiting the upper branch of \eqref{eq: est-thm01-tstar-do}, \eqref{eq: est-thm01-tstar-est-2}, recalling the definition of $ t^\ast $ and the non-expansivity of the norms, we deduce
\begin{equation}\label{eq: est-thm01-tstar-tri}
	\Vert u(t)\Vert_{\infty}\leq 
	\begin{cases}
	K \left( t^{-\frac{N}{2q_0+N(m_2-1)}} \Vert u_0\Vert_{q_0}^{\frac{2q_0}{2q_0+N(m_2-1)}} + \Vert u_0\Vert_{q_0} \right) & \forall t \in\left( 0 , t^\ast \right] , \\
	1 & \forall t \in \left( t^\ast , 2 K^\prime \left\| u_0 \right\|_{q_0}^{\frac{2q_0}{N}} \right) , \\
	K^{\prime\prime} \left( t^{-\frac{N}{2q_0+N(m_1-1)}} \ \Vert u_0\Vert_{q_0}^{\frac{2q_0}{2q_0+N(m_1-1)}} + \Vert u_0\Vert_{q_0} \right) & \forall t \ge 2 K^\prime \left\| u_0 \right\|_{q_0}^{\frac{2q_0}{N}} .
\end{cases}
\end{equation}
A straightforward computation shows that \eqref{eq: est-thm01-tstar-tri} is implied by 
\begin{equation}\label{eq: est-thm01-tstar-quater}
	\Vert u(t)\Vert_{\infty} \leq 
	\begin{cases}
	K^{\prime\prime\prime} \left( t^{-\frac{N}{2q_0+N(m_2-1)}} \left\Vert u_0 \right\Vert_{q_0}^{\frac{2q_0}{2q_0+N(m_2-1)}} + \Vert u_0\Vert_{q_0} \right) &  \forall t \in \left( 0 , 2 K^\prime \left\| u_0 \right\|_{q_0}^{\frac{2q_0}{N}} \right) , \\
	K^{\prime\prime} \left( t^{-\frac{N}{2q_0+N(m_1-1)}} \left\Vert u_0 \right\Vert_{q_0}^{\frac{2q_0}{2q_0+N(m_1-1)}} + \Vert u_0\Vert_{q_0} \right) & \forall t \ge 2 K^\prime \left\| u_0 \right\|_{q_0}^{\frac{2q_0}{N}} ,
\end{cases}
\end{equation}
up to choosing 
$$ K^{\prime\prime\prime}:=2^{\frac{2q_0+Nm_2}{2q_0+N(m_2-1)}} \, K \, . $$ 
It is then easy to check that from \eqref{eq: est-thm01-tstar-quater} estimate \eqref{thm01-smooth-est} follows provided one relabels the multiplicative constant $K$.

Let us now suppose that
\begin{equation}\label{eq: est-thm01-tstar-q0-opp}
\| u_0 \|_{q_0} > \frac{1}{2K}
\end{equation}   
instead. In this case we can assume, with no loss of generality, the validity of
\begin{equation}\label{eq: est-thm01-tstar-est-3}
t^\ast >  K^\prime \left\| u_0 \right\|_{q_0}^{\frac{2q_0}{N}} \, . 
\end{equation}
Indeed, if \eqref{eq: est-thm01-tstar-est-1} holds then the above argument leads to \eqref{eq: est-thm01-tstar-quater}, since we only used \eqref{eq: est-thm01-tstar-q0} to get \eqref{eq: est-thm01-tstar-est-1}. Hence, under \eqref{eq: est-thm01-tstar-est-3}, we can easily deduce the analogue of \eqref{eq: est-thm01-tstar-tri}:
\begin{equation}\label{eq: est-thm01-tstar-sic}
	\Vert u(t)\Vert_{\infty}\leq 
	\begin{cases}
	K \left( t^{-\frac{N}{2q_0+N(m_2-1)}} \, \Vert u_0\Vert_{q_0}^{\frac{2q_0}{2q_0+N(m_2-1)}} + \Vert u_0\Vert_{q_0} \right) & \forall t \in\left( 0 , K^\prime \left\| u_0 \right\|_{q_0}^{\frac{2q_0}{N}} \right] , \\
	1 & \forall t \in \left( K^\prime \left\| u_0 \right\|_{q_0}^{\frac{2q_0}{N}} , 2t^\ast \right) , \\
	K^{\prime\prime} \left( t^{-\frac{N}{2q_0+N(m_1-1)}} \, \Vert u_0\Vert_{q_0}^{\frac{2q_0}{2q_0+N(m_1-1)}} + \Vert u_0\Vert_{q_0} \right) & \forall t \ge 2 t^\ast \, .
\end{cases}
\end{equation}
In order to remove the presence of $ t^\ast $ in \eqref{eq: est-thm01-tstar-sic}, let us show that the lower branch can be extended down to $ t = K^\prime \left\| u_0 \right\|_{q_0}^{{2q_0}/{N}} $. In fact, by evaluating the upper branch at such time and using the non-expansivity of the norms, we get:
\begin{equation}\label{eq: est-thm01-tstar-ses}
\Vert u(t)\Vert_{\infty} \leq \frac{1}{2} +  K\,\Vert u_0\Vert_{q_0}  \quad \forall t \in \left( K^\prime \left\| u_0 \right\|_{q_0}^{\frac{2q_0}{N}} , 2t^\ast \right) .
\end{equation} 
On the other hand, in view of \eqref{eq: est-thm01-tstar-q0-opp}, estimate \eqref{eq: est-thm01-tstar-ses} implies
\begin{equation*}\label{eq: est-thm01-tstar-set}
\Vert u(t)\Vert_{\infty} \leq 2 K\,\Vert u_0\Vert_{q_0}  \quad \forall t \in \left( K^\prime \left\| u_0 \right\|_{q_0}^{\frac{2q_0}{N}} , 2t^\ast \right) ,
\end{equation*}
whence
\begin{equation}\label{eq: est-thm01-tstar-last}
	\Vert u(t)\Vert_{\infty}\leq 
	\begin{cases}
	K \left( t^{-\frac{N}{2q_0+N(m_2-1)}} \, \Vert u_0\Vert_{q_0}^{\frac{2q_0}{2q_0+N(m_2-1)}} + \Vert u_0\Vert_{q_0} \right) & \forall t \in\left( 0 , K^\prime \left\| u_0 \right\|_{q_0}^{\frac{2q_0}{N}} \right) , \\
	K^{\prime\prime\prime} \left( t^{-\frac{N}{2q_0+N(m_1-1)}} \, \Vert u_0\Vert_{q_0}^{\frac{2q_0}{2q_0+N(m_1-1)}} + \Vert u_0\Vert_{q_0} \right) & \forall t \ge K^\prime \left\| u_0 \right\|_{q_0}^{\frac{2q_0}{N}} ,
\end{cases}
\end{equation}
with $ K^{\prime\prime\prime}:= 2K \vee K^{\prime\prime} $. From \eqref{eq: est-thm01-tstar-last} the validity of \eqref{thm01-smooth-est} is immediate upon relabelling the multiplicative constant $K$. 

Finally, one can drop the assumption $ u_0 \in \LL^\infty(\Omega) $ in a standard way since the right-hand side of estimate \eqref{thm01-smooth-est} does not depend on $ \| u_0 \|_\infty $, see for instance the proof of \cite[Theorem~3.2]{GM13}. 
\end{proof}    
    
Let us now deal with the case $ m_1 \le m_2 $, for which the analysis is much simpler because $ \phi^\prime(u) $ can be bounded from below by a single power.

\begin{proof}[Proof of Theorem \ref{thm01} ($m_1 \le m_2$)]
If $ m_1 $ is smaller than or equal to $ m_2 $, it is plain that in view of \eqref{cond-phi-2}--\eqref{cond-phi-3} we have a double global lower bound on $ \phi^\prime $: 
\begin{equation}\label{cond-phi-lower-1}
\left( c_1 \wedge c_2 \right) |u|^{m_1-1}  \le \phi^\prime(u) \quad \forall u \in \mathbb{R} \, , 
\end{equation}
\begin{equation}\label{cond-phi-lower-2}
\left( c_1 \wedge c_2 \right) |u|^{m_2-1}  \le \phi^\prime(u) \quad \forall u \in \mathbb{R} \, .
\end{equation}
This means that we can proceed exactly as in the proofs of Lemma \ref{lemma: q_0-t-ast} and Corollary \ref{lemma: q_0-t-ast-ext} (with simplifications actually, since we can reason as if $ m_1=m_2=m $) to get a double smoothing estimate, with no need for introducing $ t^\ast $:
\begin{equation*}\label{case-1-smooth-1}
\Vert u(t)\Vert_{\infty} \leq K \left( t^{-\frac{N}{2q_0+N(m_1-1)}} \, \Vert u_0 \Vert_{q_0}^{\frac{2q_0}{2q_0+N(m_1-1)}} + \Vert u_0 \Vert_{q_0} \right) \quad \forall t>0 \, ,
\end{equation*}
\begin{equation*}\label{case-1-smooth-2}
\Vert u(t)\Vert_{\infty} \leq K\left( t^{-\frac{N}{2q_0+N(m_2-1)}} \, \Vert u_0 \Vert_{q_0}^{\frac{2q_0}{2q_0+N(m_2-1)}} + \Vert u_0 \Vert_{q_0} \right) \quad \forall t>0 \, .
\end{equation*}
In particular there holds
\begin{equation}\label{case-1-smooth-min}
\begin{aligned}
\Vert u(t) \Vert_{\infty}
\leq K \left[ \left( t^{-\frac{N}{2q_0+N(m_1-1)}} \, \Vert u_0 \Vert_{q_0}^{\frac{2q_0}{2q_0+N(m_1-1)}} \right) \wedge \left( t^{-\frac{N}{2q_0+N(m_2-1)}} \, \Vert u_0 \Vert_{q_0}^{\frac{2q_0}{2q_0+N(m_2-1)}} \right) + \Vert u_0 \Vert_{q_0} \right]
\end{aligned}
\end{equation}
for all $ t>0 $, and it is immediate to check that \eqref{case-1-smooth-min} is equivalent to \eqref{thm01-smooth-est} up to a different $K$.
\end{proof}     
                  
\section{Long-time estimates: proofs}\label{sect: long} 

In this section we shall complete our investigation by (mostly) addressing the long-time behaviour of solutions to \eqref{pb}, under the additional assumption that $ \Omega $ is of finite measure.
To our ends the following proposition, which also has an independent interest, is crucial.

\begin{pro}[A nonlinear Poincar\'e inequality]\label{lem: dis-zeromean}
Let $ \Omega \subset \mathbb{R}^N $ be a domain of finite measure, such that the Poincar\'e inequality \eqref{dis-poin} holds. For each $ l \ge 1/2 $, let $ \Phi_l $ be any continuous function on $ \mathbb{R} $ satisfying
\begin{equation}\label{ineq-nonlin}
\left( {C}_0 \, |y| \right)^l \le \left| \Phi_l(y) \right| \le \left( {C}_1 \, |y| \right)^l \quad \forall y \in \mathbb{R} \, ,
\end{equation} 
\begin{equation}\label{ineq-nonlin-2}
y \, \Phi_l(y) > 0 \quad \forall y \neq 0 \, ,
\end{equation} 
for some positive constants $ C_0 \le C_1 $ independent of $l$. Then there exists a positive constant $ C_P^\ast $, depending only on $ \Omega ,C_1 , C_2$, such that the nonlinear Poincar\'e inequality
\begin{equation}\label{ineq-nonlin-ineq}
\left\| \Phi_l(\xi) \right\|_2 \le C_P^\ast \left\| \nabla \Phi_l(\xi) \right\|_2 \quad \forall \xi \in \LL^1(\Omega): \, \Phi_l(\xi) \in H^1(\Omega) \, , \ \, \overline{\xi}=0 
\end{equation}
holds.
\end{pro}
\begin{proof}
It is a consequence of the methods of proof of \cite[Lemma 5.9]{GM13} and \cite[Lemma 5.6]{GMP13}, hence we shall be concise and only point out the main differences (the present result is stated under more general assumptions).

If, by contradiction, the assertion is false, then there exist a sequence of numbers $ \{ l_n \} \subset [1/2,\infty) $ and a sequence of nontrivial functions $ \{ \xi_n \} \in \LL^1(\Omega) $ with $ \Phi_{l_n}(\xi_n) \in H^1(\Omega) $, $ \overline{\xi}_n=0 $, such that 
\begin{equation}\label{nonlin-proof-1}
\left\| \nabla \Phi_{l_{n}}(\xi_n) \right\|_2 \le \frac1n \left\| \Phi_{l_n}(\xi_n) \right\|_2 \quad \forall n \in \mathbb{N} \, .
\end{equation}
By setting $ a_n:=\left\| \Phi_{l_n}(\xi_n) \right\|_2 $, $ \Psi_n:=\Phi_{l_n}(\xi_n)/a_n $ and applying the Poincar\'e inequality \eqref{dis-poin} to the function $ f=\Psi_n $, together with \eqref{nonlin-proof-1}, it is straightforward to deduce that the sequence $ \{ \Psi_n \} $ converges in $ \LL^2(\Omega) $ to a constant $ c_0 \neq 0 $. We can assume with no loss of generality that $c_0>0$: in case $ c_0<0 $ one argues likewise in view of \eqref{ineq-nonlin-2}. 
Let us then set
$$ \mathcal{Z}_n := a_n^{-\frac{1}{l_n}} \, {\xi_n} \quad \forall n \in \mathbb{N} \, . $$
Thanks to \eqref{ineq-nonlin}--\eqref{ineq-nonlin-2}, there hold
\begin{equation}\label{nonlin-proof-2}
\left| \mathcal{Z}_n \right| \le \frac{1}{C_0} \left| \Psi_n \right|^{\frac{1}{l_n}} \quad \forall n \in \mathbb{N}
\end{equation}
and 
\begin{equation}\label{nonlin-proof-3}
\liminf_{n \to \infty} \mathcal{Z}_n \ge \ell > 0 \, , \quad \ell:= \frac{\liminf_{n \to \infty} c_0^{\frac{1}{l_n}}}{C_1}  \, . 
\end{equation}
In particular,
\begin{equation}\label{nonlin-proof-4}
\lim_{n \to \infty} \left| E_n \right| = 0 \, , \quad E_n:=\left\{ x \in \Omega : \, \mathcal{Z}_n(x) < 0 \right\} .
\end{equation}
Moreover, estimate \eqref{nonlin-proof-2} and H\"{o}lder's inequality yield
\begin{equation}\label{nonlin-proof-5}
\int_{E_n} \left| \mathcal{Z}_n(x) \right| dx \le \frac{1}{C_0} \int_{E_n} \left| \Psi_n(x) \right|^{\frac{1}{l_n}} dx \le \frac{\left| E_n \right|^{1-\frac{1}{2l_n}}}{C_0} \left\| \Psi_n \right\|_{\LL^2(E_n)}^{\frac{1}{l_n}} .
\end{equation}
Upon recalling that $ \{ l_n \} \subset [1/2,\infty) $ and that $ \{ \Psi_n \} $ converges in $ \LL^2(\Omega) $, from \eqref{nonlin-proof-3}--\eqref{nonlin-proof-5} we infer
$$ 0 = \lim_{n \to \infty} \int_{\Omega} \mathcal{Z}_n(x) \, dx = \lim_{n \to \infty} \int_{\Omega \setminus E_n} \mathcal{Z}_n(x) \, dx \ge \ell \left| \Omega \right| , $$
a contradiction.
\end{proof}

\begin{cor}\label{cor: dis-zeromean}
Let the hypotheses of Proposition \ref{lem: dis-zeromean} be fulfilled, and suppose in addition that the Gagliardo-Nirenberg-Sobolev inequalities \eqref{NGN} hold for all $ r,s $ complying with \eqref{NGN-parameters-r-s-1}--\eqref{NGN-parameters-r-s-2} and $ \vartheta=\vartheta(s,r,N) $ as in \eqref{NGN-parameters-theta}. Then the nonlinear Gagliardo-Nirenberg-Sobolev inequalities
\begin{equation*}\label{ineq-nonlin-GN}
\left\| \Phi_l(\xi) \right\|_r \le C_S^\ast \left\| \nabla \Phi_l(\xi) \right\|_2^{\vartheta} \left\| \Phi_l(\xi) \right\|_s^{1-\vartheta} \quad \forall \xi \in \LL^1(\Omega): \, \Phi_l(\xi) \in H^1(\Omega) \cap \LL^s(\Omega) \, , \ \, \overline{\xi}=0 
\end{equation*}
hold for a positive constant $ C_S^\ast $ depending on $ \Omega ,C_1,C_2$ and independent of $ l \ge 1/2 $ and $ r,s $ ranging in compact subsets of $(0,\infty)$. 
\end{cor}
\begin{proof}
It is enough to apply \eqref{NGN} to $f=\Phi_l(\xi)$ and combine it with \eqref{ineq-nonlin-ineq}: 
\[
\left\| \Phi_l(\xi) \right\|_r \le  C_S \left( \left\| \nabla{\Phi_l(\xi)} \right\|_2 + \left\| \Phi_l(\xi) \right\|_2 \right)^{\vartheta} \left\| \Phi_l(\xi) \right\|_s^{1-\vartheta} 
\le \underbrace{C_S \left( 1 + C_P^\ast \right)^{\vartheta}}_{C_S^\ast} \left\| \nabla{\Phi_l(\xi)} \right\|_2 ^{\vartheta} \left\| \Phi_l(\xi) \right\|_s^{1-\vartheta} \, .
\]
\end{proof}

In fact, as mentioned in the Introduction, when the measure of the domain is finite the Poincar\'e inequality comes as a \emph{consequence} of the validity of Gagliardo-Nirenberg-Sobolev inequalities.  

\begin{pro}[Finiteness of the measure and Poincar\'e]\label{pro-imp}
Let $ \Omega \subset \mathbb{R}^N $ be a domain that supports the Gagliardo-Nirenberg-Sobolev inequalities \eqref{NGN}. Suppose in addition that $ |\Omega|<\infty $. Then the Poincar\'e inequality \eqref{dis-poin} holds.
\begin{proof}
We shall prove that, under the running assumptions, the embedding of $ H^1(\Omega) $ into $ \LL^2(\Omega) $ is compact, the result being a direct consequence of such property. To this end, let $ u_n $ be any bounded sequence in $ H^1(\Omega) $. Thanks to \eqref{NGN} and to the finiteness of the measure, we know that there exist $ r>2 $ and $ M<\infty $ such that $ \sup_{n \in \mathbb{N}} \| u_n \|_{\LL^r(\Omega)} \le M $. Let $ \Omega_k $ be a sequence of bounded regular domains satisfying $ \Omega_k \subset \Omega_{k+1} $ and $ \Omega = \bigcup_k \Omega_k $. Because $ |\Omega|<\infty $, we have that $  \lim_{k \to \infty} | \Omega \setminus \Omega_k | = 0 $. Up to a subsequence that we do not relabel, $ u_n $ converges weakly in $ H^1(\Omega) $ (and in particular in $ \LL^r(\Omega) $) to some function $ u $. In view of Rellich's Theorem, there holds
$$ \lim_{n \to \infty } \| u_n - u \|_{\LL^2(\Omega_k)} = 0 \quad \forall k \in \mathbb{N} \, , $$
so that 
\begin{equation}\label{eq: ineq-final}
\begin{aligned}
\limsup_{n \to \infty} \| u_n - u \|_{\LL^{2}(\Omega)} \le & \limsup_{n \to \infty} \| u_n - u \|_{\LL^{2}(\Omega_k)} + \limsup_{n \to \infty} \| u_n - u \|_{\LL^{2}(\Omega \setminus \Omega_k )} \\
\le & | \Omega\setminus\Omega_k |^{\frac{1}{2}-\frac{1}{r}} \left( M + \| u \|_{\LL^r(\Omega)} \right) .
\end{aligned}
\end{equation}
By letting $ k \to \infty $ in \eqref{eq: ineq-final} we deduce that $ u_n $ converges strongly to $ u $ in $ \LL^2(\Omega) $, whence the assertion.
\end{proof}
\end{pro}

Following \cite[Sections 4, 5]{GM13}, we now distinguish between data with \emph{zero} mean and data with \emph{nonzero} mean; note that the same property holds for the corresponding solutions to \eqref{pb} thanks to \emph{mass conservation} (Proposition \ref{prop-mass-c}). From here on we shall take Proposition \ref{pro-imp} for granted, i.e.~the fact that under the assumptions of Theorems \ref{thm-asym-zero} and \ref{thm-asym-nonzero} the validity of the Poincar\'e inequality \eqref{dis-poin} is ensured.
 
\subsection{The case $ \boldsymbol{ \overline{u}_0 = 0 } $: proof of Theorem \ref{thm-asym-zero}} \label{sect: long-zero}

In order to prove Theorem \ref{thm-asym-zero}, our first aim is to show that in the case of zero-mean data the corresponding solutions satisfy better smoothing effects (in fact analogous to those associated with the Dirichlet problem, see Remark \ref{rem: dirichlet} below). Afterwards, by borrowing some ideas firstly introduced in \cite{BG05,G} and then exploited in \cite{GM13,GMP13}, we prove \emph{absolute bounds}, namely $ \LL^\infty $ estimates independent of the initial datum.

\begin{lem}\label{lem: estimates-zero}
Let the hypotheses of Theorem \ref{thm-asym-zero} be fulfilled. Then the following smoothing estimate holds: 
\begin{equation}\label{thm01-smooth-est-zero} 
	\Vert u(t)\Vert_{\infty}\leq 
	\begin{cases}
	K \, t^{-\frac{N}{2q_0+N(m_2-1)}} \, \Vert u_0\Vert_{q_0}^{\frac{2q_0}{2q_0+N(m_2-1)}} & \forall t \in\left( 0 , \left\| u_0 \right\|_{q_0}^{\frac{2q_0}{N}} \right) , \\
	K \, t^{-\frac{N}{2q_0+N(m_1-1)}} \, \Vert u_0\Vert_{q_0}^{\frac{2q_0}{2q_0+N(m_1-1)}} & \forall t\geq \Vert u_0\Vert_{q_0}^{\frac{2q_0}{N}}  ,
\end{cases}
\end{equation}
where $ K $ is a positive constant depending only on $m_1 , m_2 , c_1 , c_2 , \Omega $. 
\end{lem}
\begin{proof}
Let $ p_k $ be any increasing sequence of positive numbers such that $p_0=q_0$ and $p_{\infty}=\infty$. In view of the assumptions on $ \Omega $, we can apply Corollary \ref{cor: dis-zeromean} with the choices
\begin{equation*}\label{phi-l-k}
\Phi_l(y) \equiv \Phi_k(y) = y^{\frac{m_i+p_k-1}{2}} \quad i=1,2 \, ;
\end{equation*}
moreover, since $ \overline{u}_0 = 0 $, mass conservation guarantees that $ \overline{u}(t)=0 $ for all $ t>0 $ as well. As a consequence, there holds
\begin{equation}\label{phi-l-k-u}
\frac{\left\| u(t) \right\|_{\frac{r(m_i+p_k-1)}{2}}^{\frac{m_i+p_k-1}{\vartheta}}}{C_S^{\ast\frac{2}{\vartheta}} \left\| u(t) \right\|^{\frac{(1-\vartheta)(m_i+p_k-1)}{\vartheta}}_{\frac{s(m_i+p_k-1)}{2}} } \le \int_{\Omega} \left| \nabla\!\left(u^{\frac{m_i+p_k-1}{2}}\right)\!(x,t)\right|^2 dx \quad \forall t > 0 \, , \ \, i=1,2 \, .
\end{equation} 
Hence, in the proof of Lemma \ref{lemma: q_0-t-ast} one can apply \eqref{phi-l-k-u} to \eqref{la3.2}: as a result, a stronger version of inequality \eqref{la3.2-bis} holds (with $ C_S $ replaced by $ C_S^\ast $), that is the same as above with no integral term in the right-hand side. All of the estimates carried out in Section \ref{sect: short} can then be recomputed by taking into account such an improvement: it is not difficult, though tedious, to check that the latter lead to the analogue of \eqref{thm01-smooth-est} with no additional term in the right-hand side, namely \eqref{thm01-smooth-est-zero}.
\end{proof} 

\begin{lem}\label{pro: absb}
Let the hypotheses of Theorem \ref{thm-asym-zero} be fulfilled. Then the following absolute bound holds: 
\begin{equation}\label{absb-est}
\Vert u(t)\Vert_{\infty} \leq 
\begin{cases}
 K \, t^{-\frac{1}{m_2-1}}  & \forall t \in \left(0,1 \right) , \\
 K \, t^{-\frac{1}{m_1-1}} & \forall t \ge 1 \, ,
\end{cases}
\end{equation} 
where $ K $ is a positive constant depending only on $m_1 , m_2 , c_1 , c_2 , \Omega $. 
\end{lem}
\begin{proof}
Let us consider first the case $ m_1>m_2 $. Let $ t^\ast $ be defined as in \eqref{def-t-ast}, and suppose as a first step that $ t^\ast \in (0,\infty) $. With no loss of generality we can take again $ u_0 \in \LL^\infty(\Omega) $ and set $ M:=\| u_0 \|_\infty  > 1 $, so that \eqref{upperboundphiprime} holds. If we multiply the differential equation in \eqref{pb} by $u^{q_0-1}$ (let $ q_0 > 1 $), integrate in $\Omega$ and exploit \eqref{upperboundphiprime}, we obtain: 
\begin{equation}\label{eq: diff-q}
\begin{aligned}
\frac{d}{dt}\int_{\Omega}\left| u(x,t) \right|^{q_0} dx = & -q_0(q_0-1)\int_{\Omega} \left|u(x,t)\right|^{q_0-2} \phi'(u(x,t)) \left|\nabla u(x,t) \right|^2 dx \\ 
\leq & -\frac{c}{M^{m_1-m_2}} \int_{\Omega} \left| u(x,t) \right|^{m_1+q_0-3} \left| \nabla u(x,t) \right|^2 dx \\
=  & -\frac{4 c}{\left( m_1+q_0-1 \right)^2 M^{m_1-m_2}} \int_{\Omega} \left| \nabla\! \left( u^{\frac{m_1+q_0-1}{2}} \right)\! (x,t) \right|^2 dx \, .
\end{aligned}
\end{equation} 
By arguing as in the proof of Lemma \ref{lem: estimates-zero}, in the r.h.s.~of \eqref{eq: diff-q} we can apply the nonlinear Poincar\'e inequality \eqref{ineq-nonlin-ineq} with the choice $ \Phi_{l}(y) = y^{(m_1+q_0-1)/{2}} $, which yields
\begin{equation}\label{eq: diff-q-ineq}
\frac{d}{dt}\int_{\Omega}\left| u(x,t) \right|^{q_0} dx \leq - \frac{4 c}{|\Omega|^{\frac{m_1-1}{q_0}} \, C_P^{\ast 2} \left( m_1+q_0-1 \right)^2 M^{m_1-m_2}} \left( \int_{\Omega}\left| u(x,t) \right|^{q_0} dx \right)^{\frac{m_1+q_0-1}{q_0}} ;
\end{equation}
by integrating \eqref{eq: diff-q-ineq} we end up with  
\begin{equation}\label{eq: diff-q-ine2}
\Vert u(t)\Vert_{q_0} \leq \left( \frac{C}{M^{m_1-m_2}} \, t + \Vert u_0\Vert_{q_0}^{1-m_1} \right)^{-\frac{1}{m_1-1}} \quad \forall t \ge 0 \, , \quad  C := \frac{4 \, c \, q_0}{|\Omega|^{\frac{m_1-1}{q_0}} \, C_P^{\ast 2} \, (m_1-1) \left( m_1+q_0-1 \right)^2} \, . 
\end{equation} 
In particular, \eqref{eq: diff-q-ine2} implies
\begin{equation}\label{eq: diff-q-ine3}
\Vert u(t)\Vert_{q_0} \leq C^{-\frac{1}{m_1-1}} \, \| u_0 \|^{\frac{m_1-m_2}{m_1-1}}_{\infty^{\phantom{a}}} \, t^{-\frac{1}{m_1-1}} \quad  \forall t > 0 \, .
\end{equation} 
Now we notice that, by means of the same methods of proof as in Section \ref{sect: short}, the upper branch of estimate \eqref{thm01-smooth-est-zero} in fact holds up to $ t=t^\ast $. As a consequence, thanks to a $ t/2 $-shift applied to the latter plus estimate \eqref{eq: diff-q-ine3} evaluated at $ t/2 $, we infer
\begin{equation}\label{eq: diff-q-ine4}
\Vert u(t)\Vert_{\infty} \le C^\prime \, t^{-\frac{2q_0+N(m_1-1)}{(m_1-1)[2q_0+N(m_2-1)]}} \, \Vert u_0 \Vert_{\infty^{\phantom{a}}}^{\frac{2q_0(m_1-m_2)}{(m_1-1)[2q_0+N(m_2-1)]}} \quad \forall t \in \left( 0 , t^\ast \right) ,
\end{equation}  
where $ C^\prime>0 $ is a suitable constant depending on $ m_1,m_2,N,q_0,C $ and the constant $ K $ from \eqref{thm01-smooth-est-zero}. It is straightforward to check that a routine iteration of \eqref{eq: diff-q-ine4} (which still exploits a $t/2$-shift argument) yields
\begin{equation}\label{smoothing-absb}
\Vert u(t) \Vert_\infty \le K \, t^{-\frac{1}{m_2-1}} \quad \forall t \in (0,t^\ast)
\end{equation} 
for some $ K>0 $ as in the statement (the role of $q_0$ here is inessential), which will not be relabelled from here on. For $ t>t^\ast $ one can reason exactly as if $ \phi(u) $ is of porous medium type with $ m=m_1 $ (recall the beginning of the proof of Theorem \ref{thm01}, case $ m_1>m_2 $). This gives rise to the analogue of \eqref{smoothing-absb}, namely 
\begin{equation}\label{smoothing-absb-2}
\Vert u(t) \Vert_\infty \le K \left(t-t^\ast\right)^{-\frac{1}{m_1-1}} \quad \forall t > t^\ast \, .
\end{equation}
Thanks to \eqref{smoothing-absb}, by the definition of $ t^\ast $, we can therefore deduce that 
\begin{equation}\label{abs-estimate-t-ast}
t^\ast \le T:=K^{m_2-1} \, . 
\end{equation}
In particular, by combining \eqref{smoothing-absb-2} with \eqref{abs-estimate-t-ast} we get 
\begin{equation}\label{smoothing-absb-3}
\Vert u(t) \Vert_\infty \le K \, t^{-\frac{1}{m_1-1}} \quad \forall t > 2T \, ;
\end{equation}
hence, upon collecting \eqref{smoothing-absb}, \eqref{abs-estimate-t-ast}, \eqref{smoothing-absb-3} and the non-expansivity of the norms, we end up with 
\begin{equation}\label{smoothing-part-1}
\Vert u(t)\Vert_{\infty} \leq 
\begin{cases}
 K \, t^{-\frac{1}{m_2-1}}  & \forall t \in \left(0,t^\ast \right] , \\
 1 & \forall t \in \left( t^\ast , 2T \right) , \\
 K \, t^{-\frac{1}{m_1-1}} & \forall t \ge 2T \, .
\end{cases}
\end{equation} 
By arguing as in the proof of Theorem \ref{thm01} (case $ m_1>m_2 $), up to choosing a larger constant $K$ it is apparent that estimate \eqref{smoothing-part-1} is implied by
\begin{equation*}\label{smoothing-part-2}
\Vert u(t)\Vert_{\infty} \leq 
\begin{cases}
 K \, t^{-\frac{1}{m_2-1}}  & \forall t \in \left(0,2T \right) , \\
 K \, t^{-\frac{1}{m_1-1}} & \forall t \ge 2T \, ,
\end{cases}
\end{equation*} 
whence \eqref{absb-est}.

In the case where $ t^\ast=0 $, by arguing as in previous computations it is direct to see that \eqref{smoothing-absb-3} holds for all $ t>0 $: since $ m_1>m_2 $, such an estimate trivially implies \eqref{absb-est}. In the case where $ t^\ast=\infty $, clearly \eqref{smoothing-absb} holds for all $ t>0 $, from which again \eqref{absb-est} follows. 

Let us finally discuss the case $ m_1 \le m_2 $: as remarked in the corresponding proof of Theorem \ref{thm01}, both the lower bound \eqref{cond-phi-lower-1} and \eqref{cond-phi-lower-2} on $ \phi^\prime $ hold. In particular, by reasoning as in the proof of Lemma \ref{lem: estimates-zero}, we can deduce the validity of both the estimate
\begin{equation}\label{smoothing-easy-1}
\Vert u(t) \Vert_\infty \le K \, t^{-\frac{N}{2q_0+N(m_1-1)}} \, \Vert u_0\Vert_{q_0}^{\frac{2q_0}{2q_0+N(m_1-1)}} \quad \forall t >0
\end{equation}
and
\begin{equation}\label{smoothing-easy-2}
\Vert u(t) \Vert_\infty \le K \, t^{-\frac{N}{2q_0+N(m_2-1)}} \, \Vert u_0\Vert_{q_0}^{\frac{2q_0}{2q_0+N(m_2-1)}} \quad \forall t>0 \, . 
\end{equation}
Moreover, by exploiting again \eqref{cond-phi-lower-1}--\eqref{cond-phi-lower-2} and arguing similarly to the first part of the proof (with simplifications since we are basically in the single-power case), we obtain the $ \LL^{q_0} $-absolute bounds
\begin{equation}\label{smoothing-easy-3}
\Vert u(t) \Vert_{q_0} \leq K \, t^{-\frac{1}{m_1-1}} \quad \textrm{and} \quad \Vert u(t)\Vert_{q_0} \leq K \, t^{-\frac{1}{m_2-1}} \quad \forall t>0 \, . 
\end{equation}
By gathering \eqref{smoothing-easy-1}--\eqref{smoothing-easy-3} through one $t/2$-shift step, we deduce the $ \LL^\infty $-absolute bounds
\begin{equation}\label{smoothing-easy-4}
\Vert u(t) \Vert_{\infty} \leq K \, t^{-\frac{1}{m_1-1}} \quad \textrm{and} \quad \Vert u(t)\Vert_{\infty} \leq K \, t^{-\frac{1}{m_2-1}} \quad \forall t>0 \, . 
\end{equation}
It is then straightforward to check that \eqref{smoothing-easy-4} is equivalent to \eqref{absb-est}, since $ m_1 \le m_2 $. 
\end{proof}

We are now in position to prove Theorem \ref{thm-asym-zero}. 
\begin{proof}[Proof of Theorem \ref{thm-asym-zero}]
Let us start again from the case $ m_1>m_2 $. If we combine estimate \eqref{thm01-smooth-est-zero} with \eqref{absb-est}, we obtain:   
\begin{equation}\label{smoothing-part-3}
\Vert u(t)\Vert_{\infty} \leq 
\begin{cases}
 K \left( t^{-\frac{N}{2q_0+N(m_2-1)}} \, \Vert u_0 \Vert_{q_0^{\phantom{a}}}^{\frac{2q_0}{2q_0+N(m_2-1)}} \wedge t^{-\frac{1}{m_2-1}} \right) & \forall t \in \left(0, \| u_0 \|^{\frac{2q_0}{N}}_{q_0^{\phantom{a}}} \wedge 1 \right) , \\
 K \left( t^{-\frac{N}{2q_0+N(m_1-1)}} \, \Vert u_0 \Vert_{q_0^{\phantom{a}}}^{\frac{2q_0}{2q_0+N(m_1-1)}} \wedge t^{-\frac{1}{m_1-1}} \right) & \forall t > \| u_0 \|^{\frac{2q_0}{N}}_{q_0^{\phantom{a}}} \vee 1 \, .
\end{cases}
\end{equation} 
Let us first deal with the case $ \| u_0 \|_{q_0} \le 1 $. Up to a different multiplicative constant $K$, under such assumption it is not difficult to check that \eqref{smoothing-part-3} is equivalent to 
\begin{equation}\label{smoothing-part-case1}
\Vert u(t)\Vert_{\infty} \leq 
\begin{cases}
K \, t^{-\frac{N}{2q_0+N(m_2-1)}} \, \Vert u_0 \Vert_{q_0^{\phantom{a}}}^{\frac{2q_0}{2q_0+N(m_2-1)}} & \forall t \in \left( 0, \| u_0 \|^{\frac{2q_0}{N}}_{q_0^{\phantom{a}}} \right) , \\ 
K \, t^{-\frac{N}{2q_0+N(m_1-1)}} \left( t + \| u_0 \|_{q_0}^{1-m_1} \right)^{-\frac{2q_0}{(m_1-1)[2q_0 + N(m_1-1)]}} & \forall t > 1 \, .
\end{cases}
\end{equation}  
As concerns the upper branch, it is enough to compare the two time powers involved in the minimum in \eqref{smoothing-part-3} both as $ t \downarrow 0 $ and at $ t=\| u_0 \|^{{2q_0}/{N}}_{q_0} $: one sees that the first one is always smaller. On the other hand, by means of the change of variables $ \tau = \| u_0 \|_{q_0}^{m_1-1} \, t $, one can show that the lower branches of \eqref{smoothing-part-3} and \eqref{smoothing-part-case1} are indeed equivalent. We are therefore left with providing an estimate in the region 
\begin{equation*}\label{smoothing-part-region-1} 
\| u_0 \|^{\frac{2q_0}{N}}_{q_0^{\phantom{a}}} \le t \le 1 \, .
\end{equation*} 
Here we have to exploit the lower branch of \eqref{thm01-smooth-est-zero} and the upper branch of \eqref{absb-est}, which yield
\begin{equation}\label{smoothing-part-5}
\| u(t) \|_\infty \le  K \left( t^{-\frac{N}{2q_0+N(m_1-1)}} \, \Vert u_0 \Vert_{q_0^{\phantom{a}}}^{\frac{2q_0}{2q_0+N(m_1-1)}} \wedge t^{-\frac{1}{m_2-1}} \right) \quad \forall t \in \left[ \| u_0 \|^{\frac{2q_0}{N}}_{q_0^{\phantom{a}}} , 1 \right] .
\end{equation} 
By comparing the two time powers of \eqref{smoothing-part-5} at the extremals $ t=\| u_0 \|^{{2q_0}/{N}}_{q_0}  $ and $ t=1 $, it is straightforward to show that the first one is always smaller, so that \eqref{smoothing-part-5} is in fact equivalent to 
\begin{equation}\label{smoothing-part-6}
\| u(t) \|_\infty \le  K \, t^{-\frac{N}{2q_0+N(m_1-1)}} \, \Vert u_0 \Vert_{q_0^{\phantom{a}}}^{\frac{2q_0}{2q_0+N(m_1-1)}} \quad \forall t \in \left[ \| u_0 \|^{\frac{2q_0}{N}}_{q_0^{\phantom{a}}} , 1 \right] .
\end{equation} 
Clearly, \eqref{smoothing-part-case1} and \eqref{smoothing-part-6} give \eqref{smoothing-asym-zero-1}. 

We finally deal with the case $ \| u_0 \|_{q_0} > 1 $. By reasoning in a similar way to above, one can check that \eqref{smoothing-part-3} is now the same as 
\begin{equation}\label{smoothing-part-case2}
\Vert u(t)\Vert_{\infty} \leq
\begin{cases}
K \, t^{-\frac{N}{2q_0+N(m_2-1)}} \left( t + \| u_0 \|_{q_0}^{1-m_2} \right)^{-\frac{2q_0}{(m_2-1)[2q_0 + N(m_2-1)]}} &  \forall t \in \left( 0, 1 \right) ,  \\
K \, t^{-\frac{1}{m_1-1}} & \forall t > \| u_0 \|^{\frac{2q_0}{N}}_{q_0^{\phantom{a}}} \, ,
\end{cases}
\end{equation} 
up to a different constant $K$. In the intermediate region 
\begin{equation*}\label{smoothing-part-region-2}
1 \le t \le \| u_0 \|^{\frac{2q_0}{N}}_{q_0^{\phantom{a}}} \, ,
\end{equation*} 
by combining the upper branch of \eqref{thm01-smooth-est-zero} with the lower branch of \eqref{absb-est} we get 
\begin{equation}\label{smoothing-part-5-bis-bis}
\| u(t) \|_\infty \le  K \left( t^{-\frac{N}{2q_0+N(m_2-1)}} \, \Vert u_0 \Vert_{q_0^{\phantom{a}}}^{\frac{2q_0}{2q_0+N(m_2-1)}} \wedge t^{-\frac{1}{m_1-1}} \right) \quad \forall t \in \left[ 1 , \| u_0 \|^{\frac{2q_0}{N}}_{q_0^{\phantom{a}}} \right] .
\end{equation} 
By comparing the two time powers of \eqref{smoothing-part-5-bis-bis} at the extremals $ t=1 $ and $ t=\| u_0 \|_{q_0}^{{2q_0}/N} $, it is direct to show that the latter is equivalent to 
\begin{equation*}\label{smoothing-part-5-oltem}
\| u(t) \|_\infty \le  K \, t^{-\frac{1}{m_1-1}} \quad \forall t \in \left[ 1 , \| u_0 \|^{\frac{2q_0}{N}}_{q_0^{\phantom{a}}} \right] ,
\end{equation*} 
which, together with \eqref{smoothing-part-case2}, gives rise to \eqref{smoothing-asym-zero-2}. 

In order to complete the proof, we are left with addressing the case $ m_1 \le m_2 $. Actually in the above computations we never used the fact that $  m_1>m_2 $, so they also hold for $ m_1 \le m_2 $. Hence, we just need to make sure that \eqref{smoothing-asym-zero-1}--\eqref{smoothing-asym-zero-2} are comparable to
\begin{equation}\label{eq: est-case-easy} 
t^{-\frac{N}{2q_0+N(m_2-1)}} \, \Vert u_0 \Vert_{q_0^{\phantom{a}}}^{\frac{2q_0}{2q_0+N(m_2-1)}} \wedge t^{-\frac{1}{m_2-1}} \wedge  t^{-\frac{N}{2q_0+N(m_1-1)}} \, \Vert u_0 \Vert_{q_0^{\phantom{a}}}^{\frac{2q_0}{2q_0+N(m_1-1)}} \wedge t^{-\frac{1}{m_1-1}} \, ,
\end{equation}
namely the best estimate one can get from Lemmas \ref{lem: estimates-zero}--\ref{pro: absb} (recall in particular the end of proof of Lemma \ref{pro: absb}). To this end, because we only deal with time-power functions, in the case $ \| u_0 \|_{q_0} \le 1 $ it is enough to compare \eqref{smoothing-asym-zero-1} and \eqref{eq: est-case-easy} at
$$ t \downarrow 0 \, , \quad t=\| u_0 \|_{q_0^{\phantom{a}}}^{\frac{2q_0}{N}} \, , \quad t=1 \, , \quad t=\| u_0 \|_{q_0^{\phantom{a}}}^{1-m_1} \, , \quad t=\| u_0 \|_{q_0^{\phantom{a}}}^{1-m_2} \, , \quad t\to\infty \, ,$$
whereas in the case $ \| u_0 \|_{q_0}>1 $ it is enough to compare \eqref{smoothing-asym-zero-2} and \eqref{eq: est-case-easy} at 
$$ t \downarrow 0 \, , \quad t=\| u_0 \|_{q_0^{\phantom{a}}}^{1-m_2}  \, , \quad t=\| u_0 \|_{q_0^{\phantom{a}}}^{1-m_1} \, , \quad t=1 \, , \quad t=\| u_0 \|_{q_0^{\phantom{a}}}^{\frac{2q_0}{N}} \, , \quad t\to\infty \, . $$
A straightforward check yields the assertion.
\end{proof}

\begin{oss}[The Dirichlet problem]\label{rem: dirichlet}\rm
As mentioned previously, the same smoothing effects as in Lemma \ref{lem: estimates-zero} also hold for solutions of the \emph{homogeneous Dirichlet} problem
\begin{equation*}\label{pb-d}
	\begin{cases}
		u_t=\Delta\phi(u) & \textrm{in } \Omega\times \mathbb{R}^+ \, , \\
		u=0 & \textrm{on } \partial\Omega\times \mathbb{R}^+ \, , \\
		u(0)=u_0  & \textrm{in } \Omega \, .
	\end{cases}
\end{equation*}
Furthermore, in such case $ \Omega $ can be \emph{any} domain of $ \mathbb{R}^N $: this is due to the fact that the Gagliardo-Nirenberg-Sobolev inequalities (let $r,s,\vartheta$ be as in \eqref{NGN-parameters-theta}--\eqref{NGN-parameters-r-s-2})
\begin{equation*}\label{NGN-diri}
\left\| f \right\|_r \le C_S \left\| \nabla{f} \right\|_2^{\vartheta(s,r,N)} \left\| f \right\|_s^{1-\vartheta(s,r,N)} 
\end{equation*}
hold for all $ f \in \mathcal{D}(\mathbb{R}^N) $, hence in the whole $ \dot{H}^1(\Omega) \cap \LL^s(\Omega) $ by density, where $ \dot{H}^1(\Omega)$ stands for the closure of $ \mathcal{D}(\Omega) $ w.r.t.~the $ \LL^2 $ norm of the gradient. As for the improved estimates of Theorem \ref{thm-asym-zero}, one should require the finiteness of the measure of $ \Omega $, or more generally the validity of a \emph{sub-Poincar\'e} inequality according to \cite[Theorem 4.1]{GM16}, from which absolute bounds follow.
\end{oss} 

\subsection{The case $ \boldsymbol{ \overline{u}_0 \neq 0}$: proof of Theorem \ref{thm-asym-nonzero}} \label{sect: long-non}

If the mean value of the initial datum is not zero, according to \cite[Section 5]{GM13}, the first step in order to understand the long-time behaviour of the solution is the proof of the uniform convergence to the mean value itself (which, we recall, is preserved in view of Proposition \ref{prop-mass-c}).
\begin{lem}\label{lem: smooth-rel}
Let the hypotheses of Theorem \ref{thm01} be fulfilled. Suppose moreover that $ \Omega $ is of finite measure and that $ \overline{u}_0 \neq 0 $. Then the following smoothing estimate holds: 
\begin{equation}\label{eq: smooth}
 \left\| u(t)-\overline{u}_0 \right\|_\infty \le {K}_0 \, t^{-\frac{N}{2}} \left\| u_0 - \overline{u}_0 \right\|_{1} \quad \forall t \ge 1 \,  ,
\end{equation}
where $ {K}_0 $ is a positive constant depending on $ \| u_0 \|_1 , |\overline{u}_0| , m_1, m_2, c_1, c_2, \Omega $, which can be assumed to be increasing w.r.t.~$ \| u_0 \|_1 $ and locally bounded w.r.t.~$ \left| \overline{u}_0 \right| >0 $.
\end{lem}
\begin{proof}
We proceed along the lines of proof of \cite[Theorem 5.10]{GM13}, hence we do not give full technical details. As usual we take $ u_0 \in \LL^\infty(\Omega) $ with no loss of generality. By virtue of the mass-conservation property ensured by Proposition \ref{prop-mass-c}, we know that $ \overline{u}(t)=\overline{u}_0 $ for all $ t>0 $. Hence, if we let e.g.~$ p \ge 2 $, multiply \eqref{pb} by $ (u(t)-\overline{u}_0)^{p-1} $ and integrate by parts in $ \Omega \times (t_1,t_2) $, we (formally) obtain
\begin{equation}\label{non-expansivity-differences}
\begin{aligned}
\left\| u(t_2) - \overline{u}_0 \right\|_{p}^{p} + p \, (p-1) \, \int_{t_1}^{t_2} \int_{\Omega} \left| u(x,t)-\overline{u}_0 \right|^{p-2} \phi^\prime(u(x,t)) \left| \nabla u(x,t)\right|^2 dx dt
=  \left\| u(t_1) - \overline{u}_0 \right\|_{p}^{p} 
\end{aligned}
\end{equation}
for all $ t_2>t_1 > 0 $. As a consequence, we also deduce that  
\begin{equation}\label{eq: non-exp-diff}
\left\| u(t_2) - \overline{u}_0 \right\|_{p} \le \left\| u(t_1) - \overline{u}_0 \right\|_{p} \quad \forall p \in [1,\infty] \, , \quad \forall t_2 > t_1 \ge 0 \, ;
\end{equation}
we shall refer to such property as \emph{non-expansivity of the differences} (in $ \LL^p(\Omega) $ between the solution and its mean value, or any constant actually), in agreement with \eqref{eq: non-exp}. Note that in fact \eqref{eq: non-exp-diff} holds for all $ p \in [1,\infty] $. The justification of the above computations follows by remarks similar to those pointed out in the proof of Lemma \ref{lemma: q_0-t-ast}. Thanks to Theorem \ref{thm01} (with $ q_0=1 $) and \eqref{cond-phi-2}--\eqref{cond-phi-3}, we have that
\begin{equation}\label{phi-low-ev}
\tilde{c} \left| u(x,t) \right|^{m_1-1} \le \phi^\prime\!\left( u(x,t) \right) \quad \textrm{for a.e. } (x,t) \in \Omega \times \left[ 1/2,\infty \right)  ,
\end{equation}
for a suitable $ \tilde{c}>0 $ depending on $ m_1,m_2,c_1,c_2,C_S,N,\| u_0 \|_1 $. Upon introducing the \emph{relative error} $ w:=u/\overline{u}_0-1 $, setting
\begin{equation*}\label{phi-p}
\Phi_p(y) := \int_0^y |r|^{\frac{p}{2}-1} \left| r + 1 \right|^{\frac{m_1-1}{2}} dr \quad \forall y \in \mathbb{R}
\end{equation*}
and combining \eqref{phi-low-ev} with \eqref{non-expansivity-differences}, we infer the inequality
\begin{equation}\label{non-expansivity-differences-3}
\left\| w(t_2) \right\|_{p}^{p} + \tilde{c} \left|\overline{u}_0\right|^{m_1-1} p \, (p-1) \, \int_{t_1}^{t_2} \int_{\Omega} \left| \nabla \Phi_p(w)(x,t) \right|^2 dx dt \le \left\| w(t_1) \right\|_{p}^{p}
\end{equation} 
for all $t_2 >t_1>1/2$. In Lemma 5.18 of \cite{GM13} it was shown that $ \Phi_p $ satisfies
\begin{equation}\label{lemma518}
\frac{\tilde{C}}{p^{1+1\vee\frac{m_1-1}{2}}} \, |y|^{\frac{p}{2}} \le \left| \Phi_p(y) \right| \le \frac{\left( R+1 \right)^{\frac{m_1-1}{2}} }{p} \, |y|^{\frac{p}{2}} \quad \forall y \in [-R,R] \, , \quad \forall R > 1 \, ,
\end{equation} 
for a suitable positive constant $ \tilde{C} $ depending only on $m_1$. It is plain that $ \Phi_p $ can be modified in $ (-\infty,-R) \cup (R,\infty) $ so as to satisfy \eqref{lemma518} in the whole $ \mathbb{R} $: we denote by $ \Phi_p^R $ such function. Still as a consequence of Theorem \ref{thm01}, we know that there exists $ R=R_0 >1 $, which depends only on the quantities $ \| u_0 \|_1 , |\overline{u}_0| , m_1 , m_2 , c_1 , c_2 , C_S , N  $, such that  
\begin{equation*}\label{lemma518-choice}
\left| w(x,t) \right| \le \frac{\left\| u(t)-\overline{u}_0 \right\|_\infty}{\left| \overline{u}_0 \right|} \le {R}_0 \quad \textrm{for a.e. } (x,t) \in \Omega \times [1/2,\infty) \, .
\end{equation*}
Hence, we are in position to apply Corollary \ref{cor: dis-zeromean} to $ \Phi_l \equiv \Phi_{p}^{R_0} $, so that \eqref{non-expansivity-differences-3} yields
%
\begin{equation}\label{non-expansivity-differences-5}
\left\| w(t_2) \right\|_{p}^{p} + D \, p \, (p-1) \, \int_{t_1}^{t_2} \frac{\big\| \Phi_{p}^{R_0}(w(t)) \big\|_r^{\frac{2}{\vartheta}}}{\big\| \Phi_{p}^{R_0}(w(t)) \big\|_s^{\frac{2(1-\vartheta)}{\vartheta}}} \, dt \le \left\| w(t_1) \right\|_{p}^{p} \, ,
\end{equation} 
where from here on $ D>0 $ stands for a generic constant depending only on $ \| u_0 \|_1 , |\overline{u}_0| , m_1 , m_2 , c_1 , c_2 , \Omega  $, whereas $ r,s \ge 2 $  will be chosen below. By exploiting \eqref{lemma518} (with $ R=R_0 $) and the non-expansivity of the differences, from \eqref{non-expansivity-differences-5} we infer 
\begin{equation*}\label{non-expansivity-differences-6}
\frac{D \, \tilde{C}^{\frac{2}{\vartheta}} \, (p-1) }{p^{\frac{2}{\vartheta} \left( \vartheta + 1 \vee \frac{m_1-1}{2} \right)-1} \left( R_0+1 \right)^{\frac{(1-\vartheta)(m_1-1)}{\vartheta}} } \, (t_2-t_1) \, \frac{\left\| w(t_2) \right\|_{\frac{r p}{2}}^{\frac{p}{\vartheta}}}{\left\| w(t_1) \right\|_{\frac{s p}{2}}^{\frac{p(1-\vartheta)}{\vartheta}}} \le \left\| w(t_1) \right\|_{p}^{p} \, .
\end{equation*}
After some algebra, it is straightforward to check that the choices $ s=2 $ and $ r=2+4/N $ entail 
\begin{equation}\label{non-expansivity-differences-7}
\left\| w(t_2) \right\|_{\frac{N+2}{N} \, p} \le D^{\frac{\log p}{p}} \, (t_2-t_1)^{-\frac{N}{(N+{2}) p}} \left\| w(t_1) \right\|_{p} \, ,
\end{equation} 
and by iterating \eqref{non-expansivity-differences-7} similarly to the proof of \cite[Theorem 5.10]{GM13} one ends up with
%
\begin{equation*}\label{non-expansivity-differences-8}
\left\| w(t) \right\|_\infty \le D \left( t - 1/2 \right)^{-\frac{N}{4}} \left\| w(1/2) \right\|_2 \le D \, t^{-\frac{N}{4}} \left\| w_0 \right\|_2  \quad \forall t \ge 1 \, .
\end{equation*} 
Estimate \eqref{eq: smooth} then just follows by standard interpolation plus $t/2$-shift arguments. One can choose the multiplicative constant to be increasing w.r.t.~$ \| u_0 \|_1 $ due to the method of proof, since the same property holds for the r.h.s.~of \eqref{thm01-smooth-est} and for $ R_0 $. As for the dependence on $ |\overline{u}_0| $, analogous remarks apply.
\end{proof}

We can finally establish \emph{exponential} uniform convergence to the mean value, in contrast with the zero-mean case. 
\begin{proof}[Proof of Theorem \ref{thm-asym-nonzero}] 
We follow closely the strategy used in the proofs of \cite[Theorems 4.3 and 5.11]{GM13}. 
Firstly, from \eqref{eq: smooth}
we know there exists $ t_0 \ge 1 $, depending on $ \| u_0 \|_1 , |\overline{u}_0| , m_1, m_2, c_1, c_2, \Omega $, such that 
\begin{equation}\label{inf-est}
\left| u(x,t) \right| \ge \frac{\left|\overline{u}_0\right|}{2} \quad \textrm{for a.e. } (x,t) \in \Omega \times \left( t_0,\infty \right) .
\end{equation} 
Upon taking advantage of \eqref{inf-est} in \eqref{non-expansivity-differences-3} (in the case $ p=2 $), the Poincar\'e inequality \eqref{dis-poin} and reasoning exactly as in the proof \cite[Theorem 4.3]{GM13}, we end up with 
\begin{equation}\label{exp-conv-L2-1}
\left\| u(t)-\overline{u}_0 \right\|_2 \le e^{-\frac{\tilde{c} \, \left| \overline{u}_0 \right|^{m_1-1}}{2^{m_1-1} C_P^2}\,(t-t_0)} \left\| u(t_0)-\overline{u}_0 \right\|_2 \quad \forall t \ge t_0 \, .
\end{equation} 
If we exploit \eqref{eq: smooth} between $ t $ and $t-1$, use \eqref{exp-conv-L2-1} and the smoothing effect \eqref{thm01-smooth-est}, we obtain: 
\begin{equation}\label{exp-conv-L2-2}
\left\| u(t)-\overline{u}_0 \right\|_\infty \le K_0 \, e^{-\frac{\tilde{c} \, \left| \overline{u}_0 \right|^{m_1-1}}{2^{m_1-1} C_P^2}\,t} \quad \forall t \ge t_0 + 1 \, ,
\end{equation} 
for another $ K_0>0 $ depending on the same quantities as the one in \eqref{eq: smooth}. In view of \eqref{exp-conv-L2-2} and the fact that $ \phi $ is $C^2$ in a neighbourhood of $ \overline{u}_0 $, we can infer that there exist positive constants $t_1=t_1(\| u_0 \|,\overline{u}_0,\phi,\Omega)$, $ K_1=K_1(\| u_0 \|,\overline{u}_0,\phi,\Omega) $ and $ M=M(|\overline{u}_0|,m_1,\tilde{c},C_P) $ such that 
\begin{equation}\label{exp-conv-L2-phi}
\phi^\prime(u(x,t)) \ge \phi^\prime\!\left( \overline{u}_0 \right) - K_1 \, e^{-M t}  \quad \forall t \ge t_1 \, .
\end{equation}
By plugging \eqref{exp-conv-L2-phi} in \eqref{non-expansivity-differences} (with $p=2$) and carrying out similar computations to those performed in the proof of \cite[Theorem 4.3]{GM13}, we infer the inequality 
\begin{equation}\label{exp-conv-L2-3}
\left\| u(t)-\overline{u}_0 \right\|_2 \le e^{-\frac{1}{C_P^2}\int_{t_1}^t \left[ \phi^\prime\!\left( \overline{u}_0 \right) - K_1 \, e^{-Ms} \right] \, ds}  \left\| u(t_1)-\overline{u}_0 \right\|_2 \quad \forall t \ge t_1 \, .
\end{equation} 
It is then apparent that \eqref{est-smooth} is a consequence of \eqref{eq: smooth} applied between $t$ and $t-1$, \eqref{exp-conv-L2-3}, again \eqref{eq: smooth} and the non-expansivity of the differences. As for the dependence of the multiplicative constant on $ \| u_0 \|_1 $ and $ |\overline{u}_0| $, the same comments as in the end of the proof of Lemma \ref{lem: smooth-rel} hold.
\end{proof}

\begin{oss}[Asymptotics for less regular $\phi$]\label{mod-cont}\rm
For simplicity, in Theorem \ref{thm-asym-nonzero} we assumed that $ \phi $ is $ C^2 $ away from $0$. However, as it can be guessed from the above proof, the result continues to hold under the milder hypothesis
$$ \int_0^1 \frac{\omega(r)}{r} \, dr < \infty \, ,  $$
where $ \omega : \mathbb{R}^+ \mapsto \mathbb{R}^+ $ is the local modulus of continuity of $ \phi^\prime $ in $ \mathbb{R}\setminus\{ 0 \} $.
\end{oss}

\noindent \textbf{Acknowledgements.} The second author wishes to thank Prof.~Giuseppe Savar\'e for some helpful discussions related to Section \ref{sect: well}. Both authors thank the ``Gruppo Nazionale per l’Analisi Matematica, la Probabilit\`a e le loro Applicazioni'' (GNAMPA, Italy) of the ``Istituto Nazionale di Alta Matema\-tica'' (INdAM, Italy) for support. They are also grateful to the anonymous referees for many valuable suggestions, which hopefully helped improve this work.

\bibliographystyle{plainnat}


\end{document}